\documentclass[11pt]{article}

\usepackage[utf8]{inputenc}
\usepackage[T1]{fontenc}
\usepackage[english]{babel}

\usepackage[all]{xy}
\usepackage{verbatim}
\usepackage{amsmath}
\usepackage{amsthm}
\usepackage{graphicx}
\usepackage{amssymb}
\usepackage{epstopdf}
\usepackage{url}
\usepackage{amsfonts}
\usepackage[colorinlistoftodos,prependcaption]{todonotes}
\usepackage{fullpage}

\usepackage{bbm}

\usepackage{tikz}
\usetikzlibrary{matrix}

\usepackage{hyperref}

\usepackage{mathtools}
\mathtoolsset{showonlyrefs}

\usepackage{enumerate}

\newcommand{\mc}{\mathcal}
\newcommand{\mb}{\mathbb}
\newcommand{\pt}{\partial}

\newcommand{\br}{\mathbb{R}}
\newcommand{\bz}{\mathbb{Z}}

\newcommand{\bt}{\mathbb{T}}

\newcommand{\eps}{\varepsilon}
\newcommand{\e}{\varepsilon}

\renewcommand{\r}{\rho}

\renewcommand{\(}{\left(}
\renewcommand{\)}{\right)}
\renewcommand{\[}{\left[}
\renewcommand{\]}{\right]}
\newcommand{\na}{\nabla}

\newcommand{\PP}{\mb{P}}
\newcommand{\s}{\mc{S}}

\newtheorem{thm}{Theorem}
\newtheorem{lem}[thm]{Lemma}

\newtheorem{prop}[thm]{Proposition}
\newtheorem{defi}[thm]{Definition}

\newtheorem{rks}[thm]{Remarks}

\newtheorem{hyp}[thm]{Assumption}
\def\be{\begin{equation}}
\def\ee{\end{equation}}
\def\bea{\begin{eqnarray}}
\def\eea{\end{eqnarray}}

\newcommand{\di}{\mathrm{d}}

\numberwithin{thm}{section}
\numberwithin{equation}{section}

\author{}

\title{}

\begin{document}
\title{A mean field approach to the quasineutral limit for the Vlasov-Poisson equation }

\author{
Megan Griffin-Pickering 
  \thanks{University of Cambridge, DPMMS Centre for Mathematical Sciences, Wilberforce Road, Cambridge CB3 0WB, UK. Email: \textsf{m.griffin-pickering@maths.cam.ac.uk}}
  \and
Mikaela Iacobelli
  \thanks{Durham University, Department of Mathematical Sciences, Lower Mountjoy, Stockton Road, Durham DH1 3LE, UK. Email: \textsf{mikaela.iacobelli@durham.ac.uk}}
}

\maketitle

\begin{abstract}
This paper concerns the derivation of the Kinetic Isothermal Euler system in dimension $d\ge1$ from an $N$-particle system of extended charges with Coulomb interaction.
This requires a combined mean field and quasineutral limit for a regularized $N$-particle system.
\end{abstract}

\section*{Introduction}

In this article, we consider the mathematical relationships between several models for plasma.
A plasma is a gas that has undergone a process of ionisation, whereby the particles making up the gas separate into electrons and positively charged ions. Plasma is sometimes referred to as the fourth state of matter and is abundant in our universe where it is found in stars, nebulae and many other astrophysical entities. On Earth, plasma occurs naturally in lightning and auroras, and has also been created artificially in many scientific and industrial applications including TV screens, fluorescent lights and nuclear fusion research. For further information on the physics and applications of plasmas, see for example \cite{Chen}.

One way to model a plasma is by using the following system of partial differential equations, known as the \emph{Vlasov-Poisson} system and written here for the case when the spatial domain is the $d$-dimensional torus $\bt^d$ for some $d \geq 1$:
\be \label{Eqn_VP}
(VP) : = \left\{ \begin{array}{ccc}\partial_t f+v\cdot \nabla_x f+ E\cdot \nabla_v f=0,  \\
E=-\nabla_x U, \\
- \Delta_x U=\rho_{f} - 1 ,\\
f\vert_{t=0}=f_{0}\ge0,\ \  \int_{\mathbb{T}^d \times \mathbb{R}^d} f_{0}\,dx\,dv=1,
\end{array} \right.
\ee
where for an arbitrary function $f$ we define the density
\be \label{Def_rho}
\r_f(t, x) = \int_{\br^d} f(t, x,v) \di v .
\ee
The system describes the evolution of the probability density function $f(t,x,v)$ of a typical electron in phase space (both position and velocity) rather than the individual trajectories of each electron. The ions are assumed to be heavy and so fixed in a uniform distribution over the course of the time interval of observation. 
The Vlasov-Poisson system consists of a Liouville equation for the density coupled in a nonlinear way with a Poisson equation describing the electric field spontaneously generated by the charged particles that constitute the plasma. This system of equations gives a good description of plasmas that are weakly collisional, under an electrostatic approximation. A plasma may be assumed to be weakly collisional if it is sufficiently hot and rarefied. In this case, close passes between particles are relatively infrequent and the main way that particles exert influence on each other is through the electromagnetic fields that they collectively generate. This is a long-range type interaction. The electrostatic assumption means that the magnetic field is assumed to be stationary in time (in fact zero in this case) and so the Liouville equation is coupled to a Poisson equation rather than the full Maxwell equations.

The motivation for this work comes from two key problems related to the Vlasov-Poisson system. The first concerns the derivation of the Vlasov-Poisson system from the underlying microscopic dynamics. From the point of view of classical physics, on a microscopic level a plasma can be thought of as a cloud of particles influencing each other through electromagnetic interaction. Mathematically these dynamics can be described by a system of ODEs derived from Newton's laws of motion. The principle behind the Vlasov-Poisson equation is that if the number of particles is sufficiently large and the system is observed at a sufficiently large scale, then the plasma looks more like a continuous distribution than a collection of discrete particles. Heuristically we would expect that by zooming out from the ODE system in an appropriate fashion (that is, by choosing a suitable scaling) one would arrive at the Vlasov-Poisson system. The goal is to show this rigorously, and we will formulate the precise mathematical meaning of this in more detail in the next section. In general this type of procedure is known as a \textit{mean field limit}.

The second problem concerns the \textit{quasineutral limit}. A plasma has a characteristic parameter called the Debye length. This may be thought of intuitively as the scale at which charge separation occurs; this means that when the plasma is observed at scales larger than the Debye length, the proportion of electrons and ions appears to be roughly even and thus the plasma locally appears approximately neutral. If the Debye length is much shorter than the typical scale at which the plasma is observed, then the plasma is called quasineutral. The assumption of quasineutrality may be safely assumed to hold for the purposes of many physical applications, and indeed some authors include it as part of the definition of a plasma (see for example \cite{Chen}), since in this regime an ionised gas will typically exhibit truly plasma-like behaviour as opposed to behaving like a neutral gas. It is therefore of interest to consider the limit in which the ratio of the Debye length to the typical observation length tends to zero. This is known as the quasineutral limit. Mathematically it corresponds to a certain singular limit for the Vlasov-Poisson system, in which the formal limiting system is the \emph{Kinetic Isothermal Euler} system:
\be \label{Eqn_KIE}
(KIE) :=\left\{ \begin{array}{ccc}\partial_t f +v\cdot \nabla_x f+ E\cdot \nabla_v f=0,  \\
E= -\nabla_x U, \\
\rho = 1,\\
f\vert_{t=0}=f_{0}\ge0,\ \  \int_{\mathbb{T}^d \times \mathbb{R}^d} f_0\,dx\,dv=1.
\end{array} \right.
\ee
We note that for mono-kinetic data the Kinetic Isothermal Euler system reduces to the incompressible Euler equation. In particular no global existence results are available in dimension $d\ge3.$

Since the Debye length is short in many applications, provided that the quasineutral limit is valid it is then reasonable to model a plasma using the KIE system rather than the Vlasov-Poisson system. This is advantageous since \eqref{Eqn_KIE} is simpler to solve numerically than \eqref{Eqn_VP}. Unfortunately this formal limit has only been justified for some special classes of initial data and is known to be false in general due to instabilities inherent to the physical system under consideration. In this article we will work within a class of data for which the quasineutral limit does hold rigorously.

Our aim in this article is to study the interplay between these two limits and ultimately derive the Kinetic Isothermal Euler system from a particle system, with the Vlasov-Poisson system as an intermediate step. In the next two sections we introduce the two limits in more detail.

\subsection*{Mean field limit for the Vlasov-Poisson equation}

The goal of a mean field limit in this context is to derive equation \eqref{Eqn_VP} from an underlying associated ODE system. To begin with, consider $N$ indistinguishable classical particles (of equal mass) interacting via Newton's laws of motion in $\br^d$. Let $x_i(t) \in \br^d$ denote the position of the $i$th particle at time $t$ and $v_i(t) \in \br^d$ its velocity. The combined $2d$-dimensional vector $(x_i,v_i)$ is referred to as the particle's position in phase space. Under an appropriate scaling, the phase space positions satisfy a system of ODEs of the form
\be \label{ODE_W}
\left \{
\begin{array}{l}
\dot{x}_i(t) = v_i , \\
 \dot{v}_i(t) =-\frac{1}{N} \sum_j \na W(x_i(t)-x_j(t)) .
\end{array} \right.
\ee
This is a Hamiltonian system. Particles are assumed to interact pairwise and the force between any pair is derived from a potential proportional to $W$. The factor $\frac{1}{N}$ comes from the rescaling of the system, which is chosen with the aim of obtaining a meaningful limit when $N$ tends to infinity. This ODE system gives a microscopic description of the system since the trajectory of each particle is tracked individually.

In many relevant physical applications the number of particles $N$ is extremely large. For example, in the case of gases or plasma $N$ may be of order $10^{23}$. An ODE system of such large size is impractical to solve even numerically. Furthermore, it is difficult to extract meaningful qualitative information about the behaviour of such a system from its microscopic description. This motivates the use of a coarser, \emph{mesoscopic} level of description. In this case the system is described by a probability density function corresponding to the distribution of a `typical' particle in phase space. In order to validate this type of description we would like to show that it arises as the $N \to \infty$ limit of the ODE system \eqref{ODE_W}. 

The first step is to formulate this limit mathematically. The solution of \eqref{ODE_W} takes values in the $N$-particle phase space $(\br^d \times \br^d)^N$, which has dimension dependent on the number of particles. Thus, as it stands, solutions corresponding to different numbers of particles cannot be directly compared. In response to this difficulty, we introduce the corresponding empirical measure
\be
\mu^N_t : = \frac{1}{N} \sum_{i=1}^N \delta_{(x_i(t), v_i(t))},
\ee
where $\delta_z$ denotes a Dirac mass centred at $z \in \br^d \times \br^d$. For all $N$ $\mu^N_t$ belongs to the space of measures on $\br^d \times \br^d$, and so it is meaningful to say that the sequence $\mu^N_t$ has a limit, for instance in the sense of weak convergence of measures.

Since $((x_i(t),v_i(t)))_{i=1}^N$ satisfy \eqref{ODE_W}, it follows that the corresponding equation for $\mu^N$ is
\be
\pt_t \mu^N_t+v\cdot \na_x\mu^N_t+F^N(t,x)\cdot \na_v \mu^N_t=0, \qquad F^N:=-\(\na W \ast_{x,v} \mu^N_t\) .
\ee
Hence if $\mu^N$ has a limit $\mu$ as $N \to \infty$ and $\mu$ may be written in the form $\mu_t(\di x \di v) = f(t,x,v) \di x \di v$, then formally we expect that $f$ should be a solution of the \emph{nonlinear Vlasov equation}
\begin{equation}
\label{VE}
(VE):= \left\{ \begin{array}{ccc}\pt_t f+v\cdot \pt_x f+ F(t,x)\cdot \pt_v f=0,  \\
F= -\na W \ast_{x} \rho \qquad \rho(t,x)=\int f(t,x,v)dv\,\qquad \int \rho(t,x)dx=1.\\
\end{array} \right.
\end{equation}
In this paper we consider the particular case of the Vlasov-Poisson equation \eqref{Eqn_VP}, in which $W$ is taken to be the Green kernel of the Laplacian on $\bt^d$. Note that by changing the sign of the force this system may be be used to describe both plasmas (electromagnetic interaction) and systems of stars (gravitational interaction). 
In this case the force kernel has a strong singularity of order $|x|^{-(d-1)}$ which makes the rigorous justification of the mean field limit challenging. 

For a more detailed introduction to mean field limits as a mathematical problem, we refer the reader to the lecture notes by Golse \cite{Golse2} and those by Neunzert \cite{Neu}, as well as Spohn \cite{Spohn} for a broad introduction to this and other topics in many particle systems. The mathematical literature on the rigorous justification of mean field limits dates back to the seventies, with works such as Neunzert and Wick \cite{NW74}, Braun and Hepp \cite{BraHe}, and Dobrushin \cite{Dob} tackling the case when the interaction force kernel $\nabla W$ is Lipschitz with respect to $x$. This regularity leads to a certain amount of stability for \eqref{VE}. However, from a physical perspective, Lipschitz regularity of the interaction force implies a relatively weak interaction between the particles. In many physically realistic models we would rather expect to see forces with a singularity at zero such as the Coulomb force described above. Forces of this type do not fall within the remit of the aforementioned results. However, in recent years progress was made on this front by Hauray and Jabin \cite{HJ}, who showed that the mean field limit holds for almost all choices of initial conditions for the ODE system \eqref{ODE_W} in the `weakly singular' case, i.e., when the force $\nabla W$ satisfies the following hypotheses: there exists a positive constant $C$ such that, for all  $x\in \mathbb{R}^d\setminus\{0\},$
$$
 \frac{|\nabla W(x)|}{|x|^\alpha}\le C, \frac{|\nabla^2 W(x)|}{|x|^{\alpha+1}}\le C \quad \ \mbox{for all}\ \alpha<1.
$$
Here `almost all' means that if the initial configurations $(x_i(0), v_i(0))_{i=1}^N$ are chosen randomly by drawing $(x_i(0), v_i(0))_{i=1}^\infty$ as independent samples from the probability distribution $f_0$ used as data for \eqref{Eqn_VP}, then the mean field limit holds with probability 1. As the authors of \cite{HJ} note, the weakly singular case corresponds to potentials $W$ that are continuous at the origin. This implies that the interaction is still relatively weak, with large deflections between pairs of particles being uncommon, and so it is reasonable to expect the mean field limit to be simpler in this case than for stronger singularities. For interactions with a stronger singularity the problem remains open for the full interaction. 

In the Lipschitz case, well-posedness for \eqref{VE} is implied by the mean-field limit results, as explained in \cite{Dob}. In more singular cases such as \eqref{Eqn_VP} the method cannot be used since no rigorous derivation result exists. However there are many papers dealing directly with the problem of existence of solutions for \eqref{Eqn_VP}. We refer in particular to the results of Ukai and Okabe \cite{uo}, Pfaffelmoser \cite{Pfa}, Schaeffer \cite{Sch} and Batt and Rein \cite{BR} concerning global-in-time classical solutions. Arsenev \cite{Ar} introduced a notion of weak solution and proved their global existence for initial data satisyfing $f_0 \in L^1_{x,v} \cap L^\infty_{x,v}$; the uniform boundedness constraint was later relaxed to $f_0 \in L^p$ for $p$ sufficiently large by Horst and Hunze \cite{HH}. This is the type of solution we will use for \eqref{Eqn_VP} and the rescaled version \eqref{Eqn_VP_eps} we will introduce later.

Progress has been made in the direction of a mean-field limit for stronger singularities by considering truncated or regularised versions of the interaction force that converge to the original force as $N$ tends to infinity. For instance, a truncation approach was used by Hauray and Jabin in \cite{HJ} to obtain similar results as for the weakly singular case. However this result covers singularities strictly weaker than that of the Coulomb force and thus does not include the Vlasov-Poisson case. More recently, there have been results by Lazarovici \cite{Laz15} and Lazarovici and Pickl \cite{LP15} that are able to cover (a modified version of) the Vlasov-Poisson case. In \cite{Laz15} Lazarovici considers a version of the interaction force regularised by (double) convolution, corresponding physically to the dynamics of a system of `extended charges' - clouds of charge with radius decreasing with $N$. Lazarovici identifies a set of initial configurations for which this modified mean field limit holds and further proves that this set is `large' in the sense described above. The results of \cite{LP15} are also able to cover a modified Vlasov-Poisson equation, using a truncation approach similar to that of \cite{HJ}. The methods used are probabilistic and aimed directly at a typicality type result, showing that the limit holds for `most' initial configurations without giving explicit conditions identifying these successful configurations. In this paper we will work in the setting of \cite{Laz15}, considering a modified Vlasov-Poisson system with force regularised by convolution. 

\subsection*{Quasineutral limit}

A key parameter for a plasma is the \emph{Debye screening length} $\lambda_D$. It is related to the physical parameters of the plasma via the formula
\be \label{Def_Deb}
\lambda_D : = \left (\frac{\epsilon_0 k_B T}{n q^2} \right )^{1/2},
\ee
where $\epsilon_0$ is the permittivity of free space, $k_B$ is the Boltzmann constant, $T$ is the temperature of the electrons, $n$ is their typical density and $q$ is the charge of one electron. The Debye length is important for describing several characteristic behaviours of plasma such as charge screening. It may be thought of intuitively as the scale at which charge separation typically occurs.

In many physical applications, the Debye length is much shorter than the typical scale $L$ at which the plasma is observed. It is therefore interesting to define the parameter $\e := \frac{\lambda_D}{L}$ and consider the limit in which $\e \to 0$. When we take the Debye length into account, in appropriate dimensionless variables the Vlasov-Poisson system becomes
\be \label{Eqn_VP_eps}
(VP)_{\eps} : = \left\{ \begin{array}{ccc}\partial_t f_\eps+v\cdot \nabla_x f_\eps+ E_\eps\cdot \nabla_v f_\eps=0,  \\
E_\eps=-\nabla_x U_\eps, \\
-\eps^2 \Delta_x U_\eps=\rho_{f_{\eps}} - 1 ,\\
f_\eps\vert_{t=0}=f_{0,\eps}\ge0,\ \  \int_{\mathbb{T}^d \times \mathbb{R}^d} f_{0,\eps}\,dx\,dv=1.
\end{array} \right. 
\ee
Formally this is the many particle limit of the $N$-particle ODE system
\be \label{Eqn_VP_N} 
(VP)_{N, \eps} =\left\{
\begin{array}{l}
\dot x_i=v_i,\\
\dot v_i= \frac{1}{N}\sum_{j\neq i}^N \eps^{-2} K(x_i-x_j),
\end{array}
\right.
\ee
where $K$ is defined by $K = \nabla_x G$, where $G$ is the Green kernel of the Laplacian on the $d$-dimensional torus. 

From \eqref{Eqn_VP_eps} we can see that mathematically the quasineutral limit $\e \to 0$ corresponds to a singular limit for the system \eqref{Eqn_VP_eps} in which the Poisson equation degenerates. The formal limit is the \emph{Kinetic Isothermal Euler} system
\be \label{Eqn_KIE2}
(KIE) : =\left\{ \begin{array}{ccc}\partial_t f +v\cdot \nabla_x f+ E\cdot \nabla_v f=0,  \\
E= -\nabla_x U, \\
\rho = 1,\\
f\vert_{t=0}=f_{0}\ge0,\ \  \int_{\mathbb{T}^d \times \mathbb{R}^d} f_0\,dx\,dv=1.
\end{array} \right. 
\ee
Here the force field $E$ may be thought of as a Lagrange multiplier or pressure term corresponding to the constraint that $\r = 1$. This system was also named \emph{Vlasov-Dirac-Benney} by Bardos \cite{Bar} and studied in \cite{BB, BN}.

The mathematical study of the quasineutral limit for the Vlasov-Poisson equation began in the nineties with the pioneering works of Brenier \cite{Br89, MR1360470, Br00} and Grenier, first with a limit involving defect measures \cite{BG94, Gr95}, then with a full justification by Grenier of the quasineutral limit for initial data with uniform analytic regularity \cite{Gr96}. Since then, results have been proven for other classes of initial data. For instance, in the one-dimensional case Han-Kwan and Hauray \cite{HKH15} have proven the quasineutral limit around a certain class of spatially homogeneous equilibria for \eqref{Eqn_KIE2} that satisfy a monotonicity and symmetry condition implying stability. More recently, the main result of Grenier \cite{Gr96} was extended to the case of exponentially small but rough perturbations of uniformly analytic data by Han-Kwan and the second author in \cite{HKI15, HKI14}, using Wasserstein stability estimates. This is the setting we will work in throughout this paper.

The work \cite{Gr96} also included a description of the so-called plasma waves, which are oscillations of the electric field in time, with frequency and amplitude of order $\e^{-1}$. This results in an oscillation of the velocity of the particles of order 1. Since these oscillations do not vanish in the limit $\e \to 0$, in order to obtain the limit $(VP)_\e \to (KIE)$ rigorously it is necessary to compensate for them by introducing corrector functions. Once these are properly defined, it can be shown that in the uniformly analytic case solutions of $(VP)_\e$ converge to a solution of $(KIE)$ in the limit $\e$, up to an oscillatory behaviour entirely characterised by the correctors. We explain this in more detail in section~\ref{Correctors}.

Proving the quasineutral limit is made challenging by instabilities in the dynamics, in particular what are known as \emph{two-stream} instabilities. The canonical example in which these arise is a solution of \eqref{Eqn_VP} consisting of two jets of electrons with sufficiently different velocities. Such solutions are unstable and numerical simulations show that over time the jets begin to twist around each other in phase space - see for instance \cite{BNR}.
As explained by Grenier \cite{Gr95}, this has severe consequences for the quasineutral limit. A rescaling argument shows that the quasineutral limit is strongly linked to large time limits for the Vlasov-Poisson system. Thus if we look at a class of data for which instabilities appear in \eqref{Eqn_VP} over time, we would expect this to cause problems for the quasineutral limit. Indeed it was shown in \cite{HKH15} that there exist choices of initial data of arbitrarily high Sobolev regularity for which the quasineutral limit is false. Thus it is necessary to work within a restricted class of initial data, such as analytic or near to analytic data, so that these instabilities do not obstruct the limit.

In the rest of this paper, we combine a mean field limit for the Vlasov-Poisson system in the sense of Lazarovici \cite{Laz15} with a quasineutral limit in the near-analytic setting of \cite{HKI15, HKI14} to obtain a derivation of the Kinetic Isothermal Euler system \eqref{Eqn_KIE2} from a particle system. We identify a relation between the physical parameters of the system that is sufficient for the limit to be valid.

\section{Statement of Results}

In this section we give a precise statement of our results and introduce the distances and other notions needed to do this. We will be considering the spatially periodic case where the spatial variable lies in the $d$-dimensional torus $\bt^d$, identified with $[- \frac{1}{2}, \frac{1}{2}]^d$ with the appropriate boundary identifications. Throughout, $| \cdot |$ will denote the Euclidean distance on $[- \frac{1}{2}, \frac{1}{2}]^d$ rather than the distance on the torus. We consider the cases $d = 1,2,3$. 

\subsection{Regularised Dynamics}

In the case $d=1$ we prove a result for the true Vlasov-Poisson equation, using the weak-strong estimates found in \cite{Loe} and \cite{HKI14}. In higher dimensions $(d=2,3)$ the singularity in the force kernel is stronger and these weak-strong estimates no longer hold. For these cases we will consider a regularisation of the Vlasov-Poisson equation as in \cite{Laz15}. The force is replaced by a version regularised by double convolution. To do this we construct a regularising sequence by taking $\chi$ to be a standard smooth non-negative radially symmetric mollifier, bounded by a constant $C$, with total mass 1 and supported in the unit ball. Using this we define the scaled mollifier
\begin{equation} \label{Def_chi}
\chi_r(x) = \frac{1}{r^d} \chi \left ( \frac{x}{r} \right ) .
\end{equation}

\begin{defi}
The regularised and scaled force is denoted by
\begin{equation} \label{Def_moll_force}
E_{\eps,r}[f] = \eps^{-2} \, \chi_r *_x \chi_r *_x K *_x \rho_f,
\end{equation}
where $K=\nabla_x G$ is the gradient of the Green kernel of the Laplacian on $\bt^d.$
\end{defi}

The scaled and regularised Vlasov-Poisson equation is then
\begin{equation} \label{Eqn_VP_reg}
(VP)_{\eps, r} = \left\{ \begin{array}{ccc}\partial_t f_{\eps,r} +v\cdot \nabla_x f_{\eps,r}+ E_{\eps,r}[f_{\eps,r}] \cdot \nabla_v f_{\eps,r}=0,  \\
f_{\eps,r}\vert_{t=0}=f_{0,\eps, r}\ge0 .
\end{array} \right. 
\end{equation}
This is the many particle ($N \to \infty$) limit of the following system of ODEs describing the evolution of a vector $[(x_i, v_i)]_{i=1}^N \in (\bt^d \times \br^d)^N$:
\be \label{Eqn_VP_N_reg} 
(VP)_{N, \eps, r} =\left\{
\begin{array}{l}
\dot x_i=v_i,\\
\dot v_i= \frac{1}{N}\sum_{j\neq i}^N \eps^{-2} \chi_r * \chi_r * K (x_i-x_j).
\end{array}
\right.
\ee

This particular choice of regularisation was introduced by Horst \cite{Horst} in the Vlasov-Maxwell context and utilised by Rein in \cite{Rein}. The benefit of this regularisation is that it preserves the property of conservation of energy observed by the original \eqref{Eqn_VP_eps} system. For \eqref{Eqn_VP_eps} the total energy is given by
\begin{align} \label{Def_VP_Energy}
\mathcal{E}(f_\e(t)) & := \frac{1}{2} \int_{\bt^d \times \br^d} f_\e |v|^2 \, dv dx + \frac{\e^2}{2} \int_{\bt^d} |\nabla_x U_\e|^2 \, dx \\
& = \frac{1}{2} \int_{\bt^d \times \br^d} f_\e |v|^2 \, dv dx + \frac{\e^{-2}}{2} \int_{(\bt^d)^2} [\r_{f_\e}(x) - 1] G(x-y) [\r_{f_\e}(y) - 1] \di x \di y .
\end{align}

The key observation is that \eqref{Eqn_VP_reg} conserves the following energy
\begin{align}
\mc{E}_{\e, r}(f_{\e,r}) &= \frac{1}{2}\int_{\bt^d \times \br^d} f_{\e,r} |v|^2 \, dv dx \\
&  + \frac{\e^{-2}}{2} \int_{\bt^{4d}} [\r_{f_{\e,r}}(z) - 1] \chi_r(z - x) G(x-y) \chi_r(w - y)[\r_{f_{\e,r}}(w) - 1] \di x \di y \di z \di w.
\end{align}
This is a regularised version of the energy of the Vlasov-Poisson system. As $r \to 0$ the regularised energy approximates the original Vlasov-Poisson energy. The regularised system can also be thought of as describing the dynamics of a collection of blobs of charge with shape $\chi_r$.

\subsection{Empirical measures}

We mentioned that the system \eqref{Eqn_VP_eps} arises as the formal $N \to \infty$ limit of the ODE system \eqref{Eqn_VP_N}, and that similarly \eqref{Eqn_VP_reg} is the $N \to \infty$ limit of \eqref{Eqn_VP_N_reg}. There is a question of how to formulate this limit mathematically since the solutions of these two systems lie in different spaces a priori. One way to compare the two systems is to look at the empirical measure induced by a solution of \eqref{Eqn_VP_N}. That is, we consider the measure
\be \label{Def_mu_noreg}
\mu^N_{\eps}(t) = \frac{1}{N} \sum_{i = 1}^N \delta_{(x_i(t), v_i(t))} ,
\ee
where $[(x_i(t), v_i(t))]_{i=1}^N$ is a solution of \eqref{Eqn_VP_N}. For the regularised case, we introduce the notation
\be \label{Def_mu}
\mu^N_{\eps, r}(t) = \frac{1}{N} \sum_{i = 1}^N \delta_{(x_i(t), v_i(t))} ,
\ee
where $[(x_i(t), v_i(t))]_{i=1}^N$ is a solution of \eqref{Eqn_VP_N_reg}. 

This can be used to give a genuine mathematical meaning to the idea that “\eqref{Eqn_VP_reg} is the limit of \eqref{Eqn_VP_N_reg}”: it means that if $[(x_i(t), v_i(t))]_{i=1}^N$ are solutions of \eqref{Eqn_VP_N_reg} for each $N$, then the corresponding empirical measures $\mu^N_{\eps, r}$ converge to a solution $f_\e$ of \eqref{Eqn_VP_reg} in the sense of weak convergence of measures. Of course this can only be expected to hold if \eqref{Eqn_VP_N_reg} is initialised with data that approximate the initial data used for \eqref{Eqn_VP_reg}, again in the sense that the empirical measures $\mu^N_{\e,r}(0)$ converge to $f_{0, \e}$ in the sense of weak convergence of measures. A common approach is to choose the initial configurations by taking $[(x_i(0), v_i(0))]_{i=1}^N$ to be $N$ independent samples from $f_{0, \e}$; this corresponds to following $N$ `typical' particles and implies convergence of the empirical measures to $f_{0,\e}$ almost surely by a law of large numbers argument. Our results are not limited to this specific case, however the conditions we end up imposing on the initial configurations are well suited to this approach. We shall examine this further in Section~\ref{sec:conc}.

\subsection{Distances}

We will measure the distance between the various solutions in the Wasserstein sense. We recall the definition below. The Wasserstein distances measure the closeness of random variables, or equivalently probability measures, in terms of possible couplings - that is, ways of realising the two random variables together on the same probability space. Each Wasserstein distance provides a metrisation of the topology of weak convergence of measures, for those measures that have enough moments for the distance to be well-defined.

\begin{defi}
Let $\mu$, $\nu$ be two probability measures on $\br^m$. A coupling $\pi$ of $\mu$ and $\nu$ is a probability measure on $\br^{2m}$ with marginals $\mu$ and $\nu$; that is, for all $A \in \mc{B}(\br^m)$
\be \label{Def_coup}
\begin{array}{cc}
\pi(A \times \br^m) & = \mu(A) \\
\pi(\br^m \times A) & = \nu(A)
\end{array}
\ee
We let $\Pi(\mu, \nu)$ denote the set of couplings of $\mu$ and $\nu$.
\end{defi}

\begin{defi}
The $p$\textsuperscript{th} Wasserstein distance $W_p$ is defined by
\be \label{Def_W}
W_p^p(\mu, \nu) = \inf_{\pi  \in \Pi(\mu, \nu)} \int_{(x,y) \in \br^m \times \br^m} |x-y|^p \, \di \pi(x,y) .
\ee
\end{defi}

In this paper we will mainly work with $W_1$ and $W_2$.

\subsection{Plasma oscillations} \label{Correctors}

To get convergence in the quasineutral limit, we need to correct for `plasma oscillations'. This is an oscillatory behaviour in the velocity variable that does not vanish as $\eps \rightarrow 0$. To deal with this we must introduce a corrector function $\mc{R}: \bt^d \rightarrow \br^d$, which we will define precisely below. We use this corrector to `filter' out the oscillations in the solutions.

\begin{defi} \label{Def_filt}
Let $\mu$ be a probability measure on $\bt^d \times \br^d$. Let $\mc{R} : \bt^d \to \br^d$ be given. The corresponding filtered measure $\tilde{\mu}$ is defined to be the measure such that
\be
\langle \tilde{\mu}, \phi \rangle = \int_{\bt^d \times \br^d} \phi(x, v - \mc{R}(x)) \mu( \di x \di v ) 
\ee
for all test functions $\phi$.
\end{defi}

The correctors we will use depend on the Debye length $\eps$ and are defined as follows. Let $f_{0, \eps}$ be some choice of initial data for the system \eqref{Eqn_VP_reg}, with distributional limit $g_0$ as $\e \to 0$. Let $g$ be a solution of the limiting system \eqref{Eqn_KIE} having $g_0$ as initial datum. We define the overall momentum densities
\begin{align} \label{Def_j}
j^{\eps}(0,x) & : = \int_{\br^d} v f_{0, \eps}(x,v) \di v \\
j(t,x) & : = \int_{\br^d} v g(t,x,v) \di v .
\end{align}

Then let
\be \label{corrector}
\mc{R}_\eps(t,x) : = \frac{1}{i} \left ( d_+ (t,x) e^{\frac{it}{\eps}} - d_{-}(t,x) e^{- \frac{it}{\eps}} \right ) ,
\ee

where $d_{\pm}$ are the solutions of:
\begin{align} \label{Def_d}
\nabla_x \times d_{\pm} &= 0 \\
\nabla_x \cdot \left ( \pt_t d_{\pm} + j \cdot \nabla_x d_{\pm} \right ) & = 0 \\ \label{Def_d_3}
d_{\pm}(0) & = \lim_{\eps \rightarrow 0} \nabla_x \cdot \left ( \frac{\eps E_{\eps}(0) \pm i j^{\eps}(0) }{2} \right ) ,
\end{align}
where $E^\e(0)$ is defined by
\be
E_\e(0) = K * (\r_{f_{0, \e}} - 1) .
\ee

\subsection{Main Results: General Configurations}

We need to impose reasonably strong conditions on $f_{0, \eps}$, the initial distributions for \eqref{Eqn_VP_eps}. To be able to state these assumptions, we first recall that we defined the energy of the Vlasov-Poisson system $\mathcal{E}(f_\e(t))$ in \eqref{Def_VP_Energy}. Secondly, following Grenier \cite{Gr96} we define the following analytic norm: for any given $\delta_0>1$,
\be \label{Def_analyticnorm}
\| g \|_{B_{\delta_0}} := \sum_{k \in \bz} | \widehat{g}(k) | \delta_0^{|k|},
\ee
where $\widehat{g}$ denotes the Fourier series of $g$ with respect to the spatial variable. We can now state the key assumptions on the initial data. We begin with the one-dimensional case:

\begin{hyp}[One dimensional case] \label{Hyp_data_1d} The data $f_{0, \eps}$ satisfy the following:
\begin{enumerate}[(i)]
\item (Analytic + perturbation) Each $f_{0,\e}$ may be written in the form $f_{0, \e} = g_{0, \e} + h_{0, \e}$, where for some $C, \delta_0 > 1$, $g_{0, \e}$ satisfies
$$
\sup_v \lVert g_{0, \e}(\cdot, v) \rVert_{B_{\delta_0}} \leq C 
$$

and $h_{0, \e}$ is small enough in Wasserstein sense:
$$
W_2(f_{0, \e}, g_{0, \e}) \leq \varphi(\e)
$$
for some function $\varphi$ decreasing to 0 sufficiently fast with $\e$. Explicitly, $\varphi$ may be taken to be
$$
\varphi(\e) = C \e^{-1} \exp ( - C \e^{-1}),
$$
for some $C$ sufficiently large.
\item (Control of support) $g_{0, \e}$ has bounded support in velocity, uniformly in $x$ and $\e$; that is, there exists $R > 0$ such that $g_{0, \e} = 0$ for all $|v| > R$ and all $x, \e$.
\item (Near-uniform mass density) The mass density associated to $g_{0, \e}$ is close to 1 in an analytic sense: there exists a constant $C$ such that
$$
\left \lVert \int g_{0, \e}(\cdot,v) \di v - 1 \right \rVert_{B_{\delta_0}} \leq C \e .
$$
\end{enumerate}
\end{hyp}

In higher dimensions we work with compactly supported data in order to get control of the mass density of the solution. However we can allow the size of the support to grow at a polynomial rate in $\e^{-1}$. These assumptions are chosen to match the results obtained in \cite{HKI15}. Although we believe that the hypothesis on the support may be slightly weakened to include densities that decay exponentially fast in velocity, achieving such extension here would go completely beyond the scope of this paper. 

The interested reader is referred to the papers \cite{HKI14, HKI15} for a discussion about possible initial data that satisfy our assumptions.

\begin{hyp}[Higher dimensions $d=2,3$] \label{Hyp_data} The data $f_{0, \eps}$ satisfy the following:
\begin{enumerate}[(i)]
\item (Uniform estimates) There exists $C_0$ independent of $\eps$ such that
\be \label{Est_Energy}
\lVert f_{0, \eps} \rVert_{L^{\infty}} \leq C_0 , \qquad \mc{E}(f_{0, \eps}) \leq C_0 .
\ee
\item (Control of support) For some $\gamma > 0$,
\be \label{Est_Supp}
f_{0, \eps}(x,v) = 0 \qquad \text{for} \quad |v| > \eps^{- \gamma} .
\ee
\item (Analytic + perturbation) $f_{0,\eps}$ may be decomposed into the form
$$
f_{0,\e} = g_{0,\e} + h_{0,\e},
$$
for some functions $g_{0, \eps}$ satisfying
$$
\sup_{\e\in (0,1)}\sup_{v \in \br^d} \, (1+|v|^2) \| g_{0,\e} (\cdot,v)\|_{B_{\delta_0}}  \leq C .
$$
\item (Perturbation in $W_2$) The functions $ h_{0, \eps}$ satisfy
$$
W_2(f_{0,\e},g_{0,\e}) \leq \varphi(\e) ,$$
where $\varphi$ decreases to 0 sufficiently fast with $\e$. Explicitly, it can be chosen to be
\begin{itemize}
\item in two dimensions, $\varphi(\e)= \exp\left[ \exp\left( - \frac{C}{\e^{2(1+ \max(\delta,\gamma))}}\right) \right]$,
for some constant $C>0$, \ $\delta>2$;

\item in three dimensions, $\varphi(\e)= \exp\left[ \exp\left( - \frac{C}{\e^{2+ \max(38,3\gamma))}}\right) \right]$,
for some constant $C>0$.
\end{itemize}
\end{enumerate}
\end{hyp}

For convenience, we will define a constant $\zeta$, which is fixed depending on the $\gamma$ chosen in \eqref{Est_Supp} and the dimension $d$:
\begin{itemize}
\item For $d=2$, we fix any $\delta > 2$ and let
\be \label{Def_z:2d}
\zeta = \max\{ \gamma, \delta \} .
\ee
\item For $d=3$ we let
\be \label{Def_z:3d}
\zeta = \max \left \{ \gamma, \frac{38}{3} \right \} .
\ee
\end{itemize}

We are now able to state our main results. 

\begin{thm}[One-dimensional case] \label{thm_main_1d}
Let $d=1$. For each $\eps > 0$, let $f_{0, \eps}$ be a choice of initial datum for the system \eqref{Eqn_VP_eps} satisfying Assumption \ref{Hyp_data}. Suppose that $g_{0, \eps}$ has a limit $g_0$ in the sense of distributions as $\e \to 0$. Then there exists $C > 0$ such that the following holds:

Let the initial configurations $[(x_i^\e(0), v_i^\e(0))]_{i=1}^N$ for the N-particle system \eqref{Eqn_VP_N} be chosen such that
\be \label{init_conv_1d}
\lim_{\e \to 0} \e^{-1} e^{\e^{-1} C} W_1(\mu^N_\e(0), f_{0,\e}) = 0 ,
\ee
where $\mu^N_\e(0)$ is the empirical measure corresponding to $[(x_i^\e(0), v_i^\e(0))]_{i=1}^N$. Let $[(x_i^\e(t), v_i^\e(t))]_{i=1}^N$ denote the solution of \eqref{Eqn_VP_N} with initial datum $[(x_i^\e(0), v_i^\e(0))]_{i=1}^N$ and let $\mu^N_{\eps}(t)$ denote the corresponding empirical measure. Let $\tilde{\mu}^N_{\eps}(t)$ denote the measures constructed by filtering $\mu^N_{\eps}(t)$ using the corrector $\mc{R}_{\eps}$ defined from $g$ in \eqref{corrector}, according to Definition \eqref{Def_filt}. Then
\be
\lim_{N \to \infty} \sup_{t \in [0,T]} W_1(\tilde{\mu}^N_{\eps}(t) , g(t)) = 0 .
\ee

\end{thm}

\begin{thm}[Higher dimensions] \label{thm_main}
Let $d = 2$ or $3$. For each $\eps > 0$, let $f_{0, \eps}$ be a choice of initial datum for the system \eqref{Eqn_VP_eps} satisfying Assumption \ref{Hyp_data}. Suppose that $g_{0, \eps}$ has a limit $g_0$ in the sense of distributions. 

Fix $T > 0$ and $\eta > 0$. Then there exists a constant $C_{T}$ and a weak solution $g(t)$ of \eqref{Eqn_KIE} with initial datum $g_0$ such that the following holds:

Recall the exponent $\zeta$ depending on $f_{0,\e}$ and defined in \eqref{Def_z:2d}-\eqref{Def_z:3d} and let $\eps = \eps_N$, $r = r_N$ be chosen such that
\be \label{r_eps}
r < e^{- C_{T} \eps^{-2 - d \zeta}} .
\ee

Let the initial configurations $[(x_i^\e(0), v_i^\e(0))]_{i=1}^N$ for the N-particle system \eqref{Eqn_VP_N_reg} be chosen such that the corresponding empirical measures satisfy, for some $\eta > 0$,
\be \label{Hyp_init_conv:thm}
\limsup_{N \rightarrow \infty} \frac{W_2(\mu^N_{\eps}(0), f_{0, \eps})}{\eps^{- \gamma} r^{1 + d/2 + \eta/2}} < \infty .
\ee

Let $[(x_i^{\e,r}(t), v_i^{\e,r}(t))]_{i=1}^N$ denote the solution of \eqref{Eqn_VP_N_reg} with initial datum $[(x_i^\e(0), v_i^\e(0))]_{i=1}^N$. Let $\mu^N_{\eps , r}(t)$ denote the empirical measure corresponding to $[(x_i^{\e,r}(t), v_i^{\e,r}(t))]_{i=1}^N$. Let $\tilde{\mu}^N_{\eps , r}(t)$ denote the measures constructed by filtering $\mu^N_{\eps, r}(t)$ using the corrector $\mc{R}_{\eps}$ defined from $g$ in \eqref{corrector}, according to Definition \eqref{Def_filt}. Then
\be
\lim_{N \to \infty} \sup_{t \in [0,T]} W_1(\tilde{\mu}^N_{\eps , r}(t) , g(t)) = 0 .
\ee
\end{thm}

\subsection{Main Results: Typicality}

We can also identify regimes in which the limit holds with probability 1 for $N$-particle configurations chosen by taking independent samples from the probability distributions with density $f_{0,\e}$.

\begin{thm}[Typicality in one dimension] \label{thm_typ_1d}
Let $d=1$ and fix $T>0$. Let $\{f_{0, \e}\}$ be a set of initial data satisfying Assumption~\ref{Hyp_data_1d} and such that $f_{0,\e}$ satisfies the same support assumption \ref{Hyp_data_1d}(ii) as $g_{0,\e}$, that is, there exists $R > 0$ such that $f_{0, \e} = 0$ for all $|v| > R$ and all $x, \e$. Assume that $g_{0,\e}$ has a limit $g_0$ in the sense of distributions as $\e \to 0$. There exists a constant $A$ such that if $\e = \e(N)$ is chosen to satisfy
\be
\e \geq \frac{A}{\log{N}},
\ee
then if the initial $N$-particle configurations $[(x_i^\e(0), v_i^\e(0))]_{i=1}^N$ are chosen by taking $N$ independent samples from $f_{0, \e}$, with probability 1 we have
\be
\lim_{N \to \infty} \sup_{t \in [0,T]} W_1(\tilde{\mu}^N_\e(t), g(t)) = 0 ,
\ee
where 
\begin{itemize}
\item $g(t)$ is a weak solution of \eqref{Eqn_KIE} with initial datum $g_0$ ;
\item $[(x_i^\e(t), v_i^\e(t))]_{i=1}^N$ is the solution of \eqref{Eqn_VP_N} with initial datum $[(x_i^\e(0), v_i^\e(0))]_{i=1}^N$ ;
\item $\mu^N_\e(t)$ denotes the empirical measure corresponding to $[(x_i^\e(t), v_i^\e(t))]_{i=1}^N$ ;
\item $\tilde{\mu}^N_\e(t)$ is the measure constructed by filtering $\mu^N_\e(t)$ using the corrector $\mc{R}_\e$ defined in \eqref{corrector}, according to Definition \eqref{Def_filt}, using the given choice of $f_{0,\e}$ and $g_0$.
\end{itemize}
\end{thm}

\begin{thm}[Typicality in higher dimensions] \label{thm_typ}
Let $d = 2$ or $3$. For each $\eps > 0$, let $f_{0, \eps}$ be a choice of initial datum for the system \eqref{Eqn_VP_eps} satisfying Assumption \ref{Hyp_data}. Suppose that $g_{0, \eps}$ has a limit $g_0$ in the sense of distributions. For fixed $T$, there exists a constants $C_T, A_T > 0$ such that the following holds:

Recall the exponent $\zeta$ depending on $f_{0,\e}$ and defined in \eqref{Def_z:2d}-\eqref{Def_z:3d} and let $\eps = \eps_N$, $r = r_N$ be chosen such that
\begin{align}\label{r_rate}
r &\geq A_T N^{- \frac{1}{d(d+2)} + \alpha} \\ \label{r_eps_main}
r & < e^{- C_{T} \eps^{-2 - d \zeta}} 
\end{align}
for some $\alpha > 0$. Then if the initial $N$-particle configurations $[(x_i^\e(0), v_i^\e(0))]_{i=1}^N$ are chosen by taking $N$ independent samples from $f_{0, \e}$, with probability 1 we have
\be
\lim_{N \to \infty} \sup_{t \in [0,T]} W_1(\tilde{\mu}^N_{\e,r}(t), g(t)) = 0 ,
\ee
where 
\begin{itemize}
\item $g(t)$ is a weak solution of \eqref{Eqn_KIE} with initial datum $g_0$ ;
\item $[(x_i^{\e,r}(t), v_i^{\e,r}(t))]_{i=1}^N$ is the solution of \eqref{Eqn_VP_N_reg} with initial datum $[(x_i^\e(0), v_i^\e(0))]_{i=1}^N$ ;
\item $\mu^N_{\e,r}(t)$ denotes the empirical measure corresponding to $[(x_i^{\e,r}(t), v_i^{\e,r}(t))]_{i=1}^N$ ;
\item $\tilde{\mu}^N_{\e,r}(t)$ is the measure constructed by filtering $\mu^N_{\e,r}(t)$ using the corrector $\mc{R}_\e$ defined in \eqref{corrector}, according to Definition \eqref{Def_filt}, using the given choice of $f_{0,\e}$ and $g_0$.
\end{itemize}
\end{thm}

We may note that, in the results above, 
the Debye length $\epsilon$ is assumed to decay logarithmically with respect to $N$.
This is a consequence of the exponential smallness assumption from Assumptions \ref{Hyp_data_1d}
and \ref{Hyp_data}. As observed in \cite{HKI14,HKI15} the exponential smallness is necessary, since the quasineutral limit can be false for polynomially small perturbations \cite{Gr99,HKH15}.

\subsection{Strategy of proof}

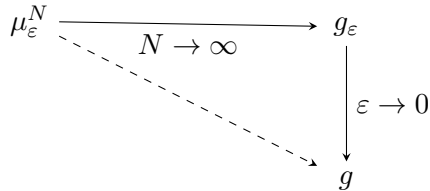
\begin{figure}[ht]
\centering
\begin{tikzpicture}
  \matrix (m) [matrix of math nodes,row sep=4em,column sep=9em,minimum width=2em]
  {
     \mu^N_{\eps} & g_{\e} \\
       & g \\};
  \path[-stealth]
    (m-1-1)             edge node [below] {$N \rightarrow \infty$} (m-1-2)
                edge [dashed] (m-2-2)
    (m-1-2) edge node [right] {$\eps \rightarrow 0$} (m-2-2);
\end{tikzpicture}
\caption{Triangular argument for the proof of Theorem~\ref{thm_main_1d}.}
 \label{fig:strat-1d}
\end{figure}

For the one dimensional case (Theorem~\ref{thm_main_1d}), the overall outline of the proof goes as follows (see Figure~\ref{fig:strat-1d}):
\begin{itemize}
\item We observe that if $\mu^N_\e(0)$ is close to $f_{0, \e}$ in Wasserstein sense, then $\mu^N_\e(0)$ is also close to $g_{0, \e}$, because of our assumption that $f_{0, \e}$ and $g_{0, \e}$ are close in Wasserstein sense. We can then use weak-strong stability estimates for \eqref{Eqn_VP_eps} around $g_\e$, the solution of of \eqref{Eqn_VP_eps} with initial datum $g_{0, \e}$. This gives the limit $\mu^N_\e \to g_\e$ for fixed $\e$. We quantify the dependence of all estimates on $\e$.
\item We use the quasineutral limit for analytic solutions proved in \cite{Gr96} to obtain the convergence $g_\e \to g$, where $g$ is a solution of \eqref{Eqn_KIE}.
\item We use the quantitative estimates to derive a condition on the initial configurations so that the full limit holds.
\end{itemize}

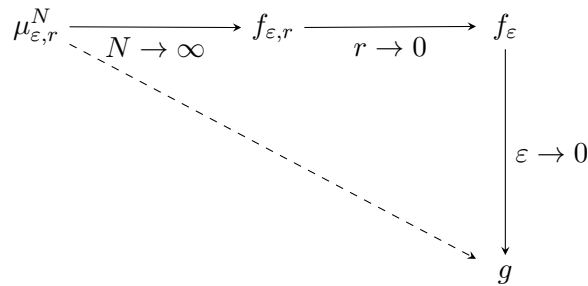
\begin{figure}[ht] 
\centering
\begin{tikzpicture}
  \matrix (m) [matrix of math nodes,row sep=7em,column sep=6em,minimum width=2em]
  {
     \mu^N_{\eps,r} & f_{\eps,r} & f_{\e} \\
      & & g \\};
  \path[-stealth]
    (m-1-1)             edge node [below] {$N \rightarrow \infty$} (m-1-2)
                edge [dashed] (m-2-3)
    (m-1-2)             edge node [below] {$r \rightarrow 0$} (m-1-3)
    (m-1-3) edge node [right] {$\eps \rightarrow 0$} (m-2-3);
\end{tikzpicture}
\caption{Triangular argument for the proof of Theorem~\ref{thm_main}.}
 \label{fig:strat-3d}
\end{figure}

In higher dimensions (Theorem~\ref{thm_main}), we no longer have weak-strong stability estimates at our disposal and so the proof is more involved (see Figure~\ref{fig:strat-3d}):
\begin{itemize}
\item We prove a mean field limit for the regularised $N$-particle system, i.e. the limit \eqref{Eqn_VP_N_reg} $\rightarrow$ \eqref{Eqn_VP_reg}, by adapting the methods of \cite{Laz15}. We quantify explicitly the dependence on $\e$.
\item We prove that the solutions of the regularised system converge to the solutions of the original system, i.e. \eqref{Eqn_VP_reg} $\rightarrow$ \eqref{Eqn_VP_eps}. Again we quantify the dependence of the rate on $\e$.
\item We use the quasineutral limit for the mean-field systems proved in \cite{HKI15} to justify the limit \eqref{Eqn_VP_eps} $\rightarrow$ \eqref{Eqn_KIE}.
\item We use the quantitative estimates to find a relation between $r = r_N$ and $\e = \e_N$ so that the full convergence \eqref{Eqn_VP_N_reg} $\rightarrow$ \eqref{Eqn_KIE} will hold.
\end{itemize}

\paragraph{A note about constants:}
Throughout this paper, we will use the notation $C$ to denote an arbitrary constant that may change from line to line. On occasion we will use subscripts to indicate the parameters on which the constant may depend, for instance $C_T$ denotes a constant depending only on $T$ and no other relevant parameters.

\section{Preliminary estimates}

\subsection{Green kernel}
We recall some properties of the Green's function $G$ for the Laplacian on the torus. The reason for this is that the potential $U_{\e}$ in \eqref{Eqn_VP_eps} may be represented in the form $G * \r_{f_\e}$. Hence when estimating $E_\e = \nabla U_\e$ it is useful to have some information about $G$. We refer to \cite{BR93} and \cite{Titch} for the results quoted below, and we add a short proof for completeness. 

\begin{lem} \label{Lem_G}
The Green kernel $G$ may be written in the form
\be \label{Def_G:3d}
G(x) = \frac{C_0}{|x|} + G_0(x) ,
\ee
when $d=3$, or
\be \label{Def_G:2d}
G(x) = C_0 \log{|x|} + G_0(x) ,
\ee
when $d=2$, for some $G_0 \in C^{\infty}$.
\end{lem}

Similarly, recalling that we defined $K = \nabla G$, we will write
\be
 K(x) = \nabla G(x) = C_1  \frac{x}{|x|^{d}} + K_0(x) ,
\ee
where $K_0$ is smooth.

\begin{proof}
We prove the case $d=3,$ the case $d=2$ being completely analogous.
Let $G$ be the Green function on the torus, 
that is
$$
\Delta G=\delta_0-1\qquad \text{on }\mathbb T^d.
$$
Then
$$
G(x)=c_d|x|^{d-2}+H(x)
$$
where $H\in C^\infty$.

\bigskip

To prove this, we look at the fundamental domain $[-1/2,1/2]^d$.
Note that $\Delta G=-1$ when $|x|>0$, hence $G$ is smooth away from the origin. So it is enough to prove the regularity of $G$ near $0$, say in $B_{1/4}$.

Consider the function $F(x):=c_d|x|^{d-2}$, and note that
$$
\Delta F=\delta_0.
$$
Set $H:=F-G$. Then
$$
\Delta H=\Delta(F-G)=1\qquad \text{in $(-1/2,1/2)^d$},
$$
so the function $H$ is $C^\infty$ inside $B_{1/4}$.
Since
$G=F+H,$
this proves the result.

\end{proof}

\subsection{Density bounds}

For our proofs we will need to control the mass densities associated with the solutions of \eqref{Eqn_VP_eps} and \eqref{Eqn_VP_reg}. In dimension $d=1$ we only require control of the mass density of $g_\e$, the solution of \eqref{Eqn_VP_eps} with analytic initial data. This control will follow directly from the results of Grenier \cite{Gr96}. In higher dimensions $d=2,3$ we will need to control solutions corresponding to rougher initial data. To be precise, we want to define $M = M_{\e, T}$ depending on $\e$ and $T$ only such that for all $r > 0$ and all $t \in [0,T]$,
\be \label{Def_M} 
\lVert \r_{f_{\eps}} \rVert_{L^{\infty}(\bt^d)} , \lVert \r_{f_{\eps, r}} \rVert_{L^{\infty}(\bt^d)} \leq M .
\ee

To show that it is possible to find such an $M$ and to get an estimate of its dependence on $\e$ we appeal to the results of \cite{HKI15}, which rely on controlling the growth of the support of $f_{\eps,r}$ in the velocity variable and are based on estimates by Batt and Rein \cite{BR}. In two dimensions we have
\begin{prop} \label{prop_dens_2d}
Let $f_{\eps}$ be a solution of \eqref{Eqn_VP_eps}, for initial data satisfying Assumption \ref{Hyp_data}. Fix $T > 0$. Then for all $\delta > 2$, there exists $C_{\delta, T}$ such that for all $t \in [0,T]$,
\be
\lVert \r_{f_{\eps}} \rVert_{L^{\infty}(\bt^d)} \leq C_{\delta, T} \, \eps^{- 2 \max \{ \delta, \gamma \}} .
\ee
\end{prop}

In three dimensions we have
\begin{prop} \label{prop_dens_3d}
Let $f_{\eps}$ be a solution of \eqref{Eqn_VP_eps}, for initial data satisfying Assumption \ref{Hyp_data}. Fix $T > 0$. Then there exists $C_T$ such that for all $t \in [0,T]$ and all $\eps \in (0,1)$ we have
\be
\lVert \r_{f_{\eps}} \rVert_{L^{\infty}(\bt^d)} \leq C_{T} \, \eps^{- \max \{ 38, 3 \gamma \}} .
\ee
\end{prop}

The proofs are given for the \eqref{Eqn_VP_eps} system, but also apply in the regularised case.

In summary, we can take 
\be \label{M_z}
M=M_{\e, T} = C_T \e^{- \zeta d},
\ee
where $\zeta$ is defined in \eqref{Def_z:2d}-\eqref{Def_z:3d}.

\subsection{Basic properties of Wasserstein distances}

We recall some useful properties of Wasserstein distances regarding their behaviour with respect to mollification. The following results follow immediately from Proposition 7.16 of \cite{Vil03}:

The Wasserstein distance may only be decreased by convolution with $\chi_r$:
\begin{lem} \label{W_moll_two}
Let $\mu$, $\nu$ be probability measures, $r>0$ any positive constant and $\chi_r$ a mollifier as defined in \eqref{Def_chi}. Then
\be \label{Est_W_conv}
W_p(\chi_r * \mu, \chi_r * \nu) \leq W_p(\mu, \nu) .
\ee
\end{lem}

The Wasserstein distance between a measure and its mollification can be controlled explicitly:
\begin{lem} \label{W_moll_one}
Let $\mu$ be a probability measure and $r> 0$. Let $\chi_r$ be a mollifier as defined in \eqref{Def_chi}. Then
\be \label{Est_W_self_conv}
W_p(\chi_r * \mu, \mu) \leq r .
\ee
\end{lem}

\subsection{Behaviour of Wasserstein distances under filtering}

We will need to account for the effect of the filtering on the Wasserstein distance. For this we quote the following lemma from \cite{HKI15}:

\begin{lem} 
\label{Lem_filt} Let $\nu_1, \nu_2$ be probability measures on $\bt^{d} \times \br^{d}$, and let $\tilde{\nu}_i$ denote $\nu_i$ filtered by a given vector field $\mc{R} : \bt^d \to \br^d$ (see Definition \ref{Def_filt}). Then
\be \label{Est_filt}
W_1(\tilde{\nu}_1, \tilde{\nu}_2) \leq (1 + \lVert \nabla_x \mc{R} \rVert_{L^{\infty}}) W_1(\nu_1, \nu_2) .
\ee
\end{lem}

In this paper we will always choose the corrector $\mc{R}_{\e}$ defined by \eqref{corrector}. In this case
\be
|\nabla_x \mc{R}_{\e}| \leq |\nabla_x d_+| + | \nabla_x d_- | .
\ee
Thus there exists $C_T$ independent of $\eps$ such that for $t \in [0,T]$,
\be
\lVert \nabla_x \mc{R}_{\e} \rVert_{L^{\infty}(\bt^d)} \leq C_T .
\ee

\section{One-dimensional case}

\subsection{Stability}

The core of the argument is a weak-strong stability estimate for the system \eqref{Eqn_VP_eps}. A suitable estimate was proved by Hauray in \cite{Hauray}. This was stated in the case $\e=1$ but is of course true for general $\e>0$. For our purposes we will need to keep track of the $\e$-dependence of the constants in the estimate. For this we will use a scaling argument as used previously in \cite{HKI14}. The results below were stated previously in \cite{HKI14} as a minor modification of the work there, however we will write the argument in full here for completeness.

\begin{thm} \label{thm_1d_stab}
Let $f_\e, \nu_\e$ be solutions of \eqref{Eqn_VP_eps} in dimension $d=1$. Suppose that $\r_{f_\e}(t) \in L^{\infty}$ for all $t \in [0,T]$. Then for all $t \in [0,T]$ we have the estimate
\be \label{est_1d_stab}
W_1(f_\e(t), \nu_\e(t)) \leq \e^{-1 }e^{\e^{-1} \alpha(t)} \, W_1(f_0, \nu_0) ,
\ee
where the exponent $\alpha$ can be taken to be
\be
\alpha(t) = \sqrt{2} t + 8 \int_0^t \lVert \r_{f_\e(s)} \rVert_{L^{\infty}} \di s .
\ee
\end{thm}

The key ingredient is the following weak-strong stability estimate from \cite{Hauray}:

\begin{thm} \label{thm_1d_Hau}
Let $f(t), \nu(t)$ be solutions of \eqref{Eqn_VP_eps} with $\e=1$ in dimension $d=1$. Then for all $t$ we have the estimate
\be \label{est_1d_Hau}
W_1(f(t), \nu(t)) \leq e^{\alpha(t)} \, W_1(f_0, \nu_0) ,
\ee
where the exponent $\alpha$ can be taken to be
\be
\alpha(t) = \sqrt{2} t + 8 \int_0^t \lVert \r_{f(s)} \rVert_{L^{\infty}} \di s .
\ee
\end{thm}

We extend this to the case of general $\e$ by scaling. As in \cite{HKI14}, we define
\be
\mc{F}_{\e}(t,x,v) := \e^{-1} f_{\e} (\e t, x, \frac{v}{\e}) . 
\ee

Observe that $\mc{F}_{\e}$ is a solution of the system
\be \label{Eqn_rescale}
\left\{ \begin{array}{ccc}\partial_t \mc{F}_\eps+v\cdot \nabla_x \mc{F}_\eps+ E_\eps\cdot \nabla_v \mc{F}_\eps=0,  \\
E_\eps=-\nabla_x U_\eps, \\
- \Delta_x U_\eps=\varrho_{\e} - 1 ,\\
\varrho_{\e} = \int_{\br^d} \mc{F}_{\e} \, \di v , \,
\int_{\br^d \times \br^d} \mc{F}_\e \di x \di v = 1 .
\end{array} 
 \right.
\ee
This is \eqref{Eqn_VP_eps} with $\e=1$, and therefore we may apply the stability estimate Theorem~\ref{thm_1d_Hau} to it. Our aim is to use this to deduce information concerning the original solution $f_\e$. To do this we need to examine the effect of this scaling on the Wasserstein distance. This is the content of the next lemma, which is proved in \cite{HKI14}.
\begin{lem} \label{lem_1d_scaleW}
Let $f_\e(t)$ be a probability density on $\br^d \times \br^d$ and let $\nu_\e(t)$ be a probability measure on $\br^d \times \br^d$. Let
\be
\mc{F}_{\e}(t,x,v) := \e^{-d} f_{\e} (\e t, x, \frac{v}{\e}) ,
\ee
and let $\mc{N}_\e(t)$ be the measure such that for any $\phi \in C^{\infty}_c(\br^d \times \br^d)$,
\be
\int_{\br^d \times \br^d} \phi(x,v) \, \di \mc{N}_\e (t) = \int_{ \br^d \times \br^d} \phi \left (x, \e v \right) \, \di \nu_{\e} (\e t) .
\ee
Then
\be
W_1(\mc{F}_\e(\e^{-1} t), \mc{N}_\e(\e^{-1} t)) \leq W_1(f_\e(t), \nu_\e(t)) \leq \e^{-1} W_1(\mc{F}_\e(\e^{-1} t), \mc{N}_\e(\e^{-1} t)) .
\ee
\end{lem}

\begin{proof}[Proof of Theorem~\ref{thm_1d_stab}]
We define $\mc{F}_\e$, $\mc{N}_\e$ as in Lemma~\ref{lem_1d_scaleW} and observe that these are solutions of \eqref{Eqn_rescale} satisfying the assumptions of Theorem~\ref{thm_1d_Hau}. Therefore
\be
W_1(\mc{F}_\e(\e^{-1} t), \mc{N}_\e(\e^{-1} t)) \leq \exp{ \left ( \sqrt{2} \e^{-1} t + 8 \int_0^{\e^{-1}t} \lVert \r_{\mc{F}_\e(s)} \rVert_{L^{\infty}} \di s \right )} \, W_1(\mc{F}_\e(0), \mc{N}_\e(0)) .
\ee

A change of variable in the exponent gives
\be
\int_0^{\e^{-1}t} \lVert \r_{\mc{F}_\e(s)} \rVert_{L^{\infty}} \di s = \e^{-1} \int_0^{t} \lVert \r_{\mc{F}_\e(\e^{-1} s)} \rVert_{L^{\infty}} \di s.
\ee
Then, noting that $\r_{\mc{F}_\e(\e^{-1} s)} = \r_{f_\e(s)}$, we deduce that
\be
\int_0^{\e^{-1}t} \lVert \r_{\mc{F}_\e(s)} \rVert_{L^{\infty}} \di s = \e^{-1} \int_0^{t} \lVert \r_{f_\e(s)} \rVert_{L^{\infty}} \di s,
\ee
and thus the full exponent is given by $\e^{-1} \alpha(t)$. We complete the proof by applying Lemma~\ref{lem_1d_scaleW}:
\begin{align}
W_1(f_\e(t), \nu_\e(t)) &\leq \e^{-1} W_1(\mc{F}_\e(\e^{-1} t), \mc{N}_\e(\e^{-1} t))  \\
&\leq \e^{-1} e^{\e^{-1} \alpha(t)} W_1(\mc{F}_\e(0), \mc{N}_\e(0)) \\
& \leq \e^{-1} e^{\e^{-1} \alpha(t)} W_1(f_\e(0), \nu_\e(0)) .
\end{align}
\end{proof}

\subsection{Existence of solutions}

We recall the following result based on Theorem 1.1.2 of \cite{Gr96}, which gives existence of solutions with bounded density under our assumptions.

\begin{thm} \label{thm_VP_exist}
Let $g_{0,\e}$ satisfy the conditions stated in Assumption~\ref{Hyp_data_1d}. Then there exist $T>0$, $\delta_1 > 1$ and for each $\e$ a unique solution $g_\e$ of \eqref{Eqn_VP_eps} with initial datum $g_{0,\e}$, such that $g_\e \in C([0,T] ; B_{\delta_1})$. Moreover there exists $C$ such that
$$
\sup_{t \in [0,T]} \left \lVert \int g_\e(t, \cdot, v) \di v - 1 \right \rVert_{B_{\delta_1}} \leq C \e,
$$
and thus these solutions have mass density bounded uniformly in $\e$ on $[0,T]$.
\end{thm}

\subsection{Quasineutral limit}

In this section we recall the quasineutral limit in the analytic case, proved by Grenier in \cite{Gr96}, Theorems 1.1.2 and 1.1.3. The theorem as stated in \cite{Gr96} actually results in $H^s$ convergence for a certain representation of the solution. The passage from this result to convergence of the solutions in Wasserstein sense follows the same procedure as detailed in \cite{HKI14}, Corollary 4.2. Altogether this results in the following:

\begin{thm} \label{thm_1d_QN}
Let $g_{0,\e}$ be a sequence of initial data satisfying the conditions given in Assumption~\ref{Hyp_data_1d}. Let $g_\e$ denote the solution of \eqref{Eqn_VP_eps} with initial datum $g_{0, \e}$, which exists on some interval $[0,T]$ uniform in $\e$ by Theorem~\ref{thm_VP_exist}. Let $\tilde{g}_\e$ denote the filtered distribution given by filtering $g_\e$ using the correctors defined in \eqref{corrector}, as described in Definition~\ref{Def_filt}. Assume that $g_{0, \e}$ has a weak limit $g_{0,0}$ in the sense of distributions as $\e \to 0$. Then there exists a solution $g_0$ of \eqref{Eqn_KIE} with initial datum $g_{0,0}$ such that
$$
\lim_{\e \to 0} \sup_{t \in [0,T]} W_1(\tilde{g}_\e(t), g_0(t)) = 0 .
$$

\end{thm}

\subsection{Completion of proof of Theorem~\ref{thm_main_1d}}

Once again we use $\tilde{\mu}$ to denote the distribution produced by filtering $\mu$ using the corrector function $\mc{R}_\e$, where we take $\mc{R}_\e$ as defined in \eqref{corrector}. Lemma~\ref{Lem_filt} gives us the following estimate between filtered and unfiltered distributions in Wasserstein sense: for all $t \leq T$, there exists a constant $C$ depending only on $T$ such that
\be \label{Est_filt_gmu}
W_1(\tilde{\mu}_\e^N(t), \tilde{g}_\e(t)) \leq C W_1(\mu_\e^N(t) , g_\e(t)) .
\ee

 \begin{proof}[End of proof of Theorem~\ref{thm_main_1d}]
 Theorem~\ref{thm_VP_exist} implies the existence of a uniform mass bound, that is, there exists $M > 0$ satisfying \eqref{Def_M}. Using the triangle inequality, the filtering estimate \eqref{Est_filt_gmu} and Theorem~\ref{thm_1d_stab}, we obtain 
\begin{align}
\sup_{t \in [0,T]} W_1(\tilde{\mu}^N_{\e} (t), g_0(t)) & \leq \sup_{t \in [0,T]} W_1(\tilde{\mu}^N_{\e} (t), \tilde{g}_\e(t)) +\sup_{t \in [0,T]} W_1(\tilde{g}_\e(t), g_0(t)) \\
&\leq \sup_{t \in [0,T]} CW_1(\mu^N_{\e} (t), g_\e(t)) +\sup_{t \in [0,T]} W_1(\tilde{g}_\e(t), g_0(t)) \\
 & \leq C \sup_{t \in [0,T]} \e^{-1} e^{\e^{-1} CMT} W_1(\mu^N_{\e} (0), g_\e(0)) +\sup_{t \in [0,T]} W_1(\tilde{g}_\e(t), g_0(t)) .
\end{align}
The second term on the right hand side converges to 0 by Theorem~\ref{thm_1d_QN}. For the first term, we estimate
\be
\e^{-1} e^{\e^{-1} CMT} W_1(\mu^N_{\e} (0), g_\e(0)) \leq \e^{-1} e^{\e^{-1} CMT} W_1(\mu^N_{\e} (0), f_{0, \e}) + \e^{-1} e^{\e^{-1} CMT} W_1(f_{0,\e}, g_{0,\e}) .
\ee
The first term converges by assumption on the initial conditions \eqref{init_conv_1d}. The second term converges by the hypotheses on the initial data $f_{0, \e}$ stated in Assumption~\ref{Hyp_data_1d}, using the fact that $W_1(f_{0,\e}, g_{0,\e}) \leq W_2(f_{0,\e}, g_{0,\e})$.
\end{proof}

\section{Mean Field Limit for the Regularised Vlasov-Poisson Equation}

In this section we prove a quantitative estimate of the rate of convergence of the empirical measure of the solution $\mu^N_{\eps,r}$ of the scaled and regularised $N$-particle system \eqref{Eqn_VP_N}, to the solution $f_{\eps,r}$ of the mean-field regularised equation \eqref{Eqn_VP_reg}, using the methods of \cite{Laz15}.

\begin{prop} \label{Prop_MFL}
Fix $T > 0$. For any small $\beta > 0$, there exists $C_{\beta, T}$ such that the following holds: Let $f_{\eps,r}$ be a solution of \eqref{Eqn_VP_reg} and $\mu^N_{\eps,r}$ be defined as in \eqref{Def_mu}. Let $\eps = \eps_N$, $r = r_N$ be chosen such that
\be \label{r_eps_2}
r < e^{- C_{\beta, T} \eps^{-2 - d \zeta}}
\ee
Assume that the initial condition for \eqref{Eqn_VP_N} is `well-placed' in the sense that there exists $\eta > \beta$ such that
\be \label{Hyp_init_conv}
\lim_{N \rightarrow \infty} \frac{W_2(\mu^N_{\eps, r}(0), f_{0, \eps})}{\eps^{- \gamma} r^{1 + d/2 + \eta/2}} = 0 .
\ee
 Then for any $\eta ' \in (\beta, \eta)$, there exists a constant $C = C(\beta, T, \eta, \eta ', \gamma, \zeta)$ such that for all $N$ sufficiently large, for all $t \leq T$ we have
\be \label{Est_Gron}
W^2_2(\mu^N_{\eps,r}(t), f_{\eps, r}(t)) \leq C r^{d+2 + \eta ' - \beta} .
\ee
\end{prop}

\subsection{Outline of strategy and anisotropic distance}

The aim is to prove a Gronwall type estimate on the Wasserstein distance between $f_{\eps,r}$ and $\mu^N_{\eps,r}$. By definition, for any coupling $\pi \in \Pi(f_{\eps,r},\mu^N_{\eps,r})$,
\be \label{Est_W_pi}
W_2^2(f_{\eps,r},\mu^N_{\eps,r}) \leq \int_{(\bt^{d} \times \br^{d})^2} |x-y|^2 + |v-w|^2 \, \di \pi(x,v,y,w).
\ee

We consider in particular couplings that evolve according to the dynamics of the respective Vlasov-Poisson equation. That is, take $\pi_0$ to be any coupling of $f_{0, \eps}$ and $\mu^N_{\eps}(0)$. Then at later times take $\pi_t$ to satisfy
\begin{equation} \label{Eqn_pi}
(\pt_t  + v \cdot \nabla_x + w \cdot \nabla_y) \pi_t + E_{\eps,r}[f_{\eps,r}](x) \cdot \nabla_v + E_{\eps,r}[\mu^N_{\eps,r}](y) \cdot \nabla_w \pi_t = 0 .
\end{equation}
In each of its variables (in phase space), $\pi_t$ evolves according to the dynamics of the corresponding Vlasov-Poisson system. That is, the $(x,v)$ part evolves along the characteristic flow induced by $f_{\e,r}$, while the $(y,w)$ part evolves along the flow induced by $\mu^N_{\e,r}$. Thus $\pi_t$ remains a coupling of $f_{\eps,r}(t)$ and $\mu^N_{\eps,r}(t)$ for all $t$ and so we may use it to construct an upper bound on the Wasserstein distance. 

The change in the first term of \eqref{Est_W_pi} (position) is controlled by the relative velocity, while the change in the second term (velocity) is controlled by the difference in the forces. We will see later that the estimates on the force terms introduce significantly larger constants than those controlling the velocity terms. In order to close the Gronwall estimate it is useful to be able to mitigate this disparity by working with an anisotropic distance of the form
\begin{equation} \label{Def_D}
D(t) = \frac{1}{2} \int_{ \bt^{2d} \times \br^{2d}} \lambda^2 |x-y|^2 + |v-w|^2 \di \pi_t(x,y,v,w) .
\end{equation}

It will suffice to control $D$ since, as long as $\frac{\lambda^2}{2} > 1$, we have
\be \label{ineq_W-D}
W_2^2(f_{\e, r}, \mu^N_{\e,r}) \leq D .
\ee
Conversely, by definition of $W_2$ we also have
\be \label{Est_infD-W}
\inf_{\pi_0} D(0) \leq \frac{1}{2} \lambda^2 W_2^2(f_{0, \e}, \mu^N_{\e} (0)),
\ee

Furthermore, when considering only the $x$ variable we can get a stronger estimate, since
\begin{align} \label{Est_Wx_D}
W^2_2(\r_{f_{\eps,r}}, \r_{\mu^N_{\eps,r}}) &= \inf_{\tilde{\pi} \in \Pi(\r_{f_{\eps,r}}, \r_{\mu^N_{\eps,r}})} \int_{ \bt^{2d}} |x - y|^2 \di \tilde{\pi}(x,y) \\
& \leq \int_{ \bt^{2d}} |x - y|^2 \di \pi_t(x,v,y,w) \\
& \leq 2 \lambda^{-2} D(t) .
\end{align}

Using the Vlasov-Poisson dynamics, we compute, for any $\alpha > 0$,
\begin{align} \label{Est_Dprime_1}
D'(t) & = \int_{ (\bt^{d} \times \br^{d})^2} \lambda^2 (x-y) \cdot (v-w) + (v-w) \cdot ( E_{\eps,r}[f_{\eps,r}](x) - E_{\eps,r}[\mu_{\eps,r}^N](y) ) \di \pi_t \\
& \leq \lambda D(t) + \frac{\alpha}{2} \int_{ (\bt^{d} \times \br^{d})^2} |v-w|^2  \di \pi_t \\
& \qquad \qquad \qquad \qquad +  \frac{1}{2 \alpha} \int_{ (\bt^{d} \times \br^{d})^2} | E_{\eps,r}[f_{\eps,r}](x) - E_{\eps,r}[\mu_{\eps,r}^N](y) |^2 \di \pi_t \\
& \leq (\lambda + \alpha) D(t) + \frac{1}{\alpha} \underbrace{\int_{ (\bt^{d} \times \br^{d})^2} | E_{\eps,r}[f_{\eps,r}](x) - E_{\eps,r}[\mu_{\eps,r}^N](x) |^2 \di \pi_t}_{= : I_1} \\
& \qquad \qquad \qquad \qquad + \frac{1}{\alpha} \underbrace{\int_{ (\bt^{d} \times \br^{d})^2 } | E_{\eps,r}[\mu_{\eps,r}^N](x) - E_{\eps,r}[\mu_{\eps,r}^N](y) |^2 \di \pi_t}_{= : I_2} .
\end{align}

\subsection{Estimation of force terms}

First we estimate $I_1$. Using that the $x$-marginal of $\pi_t$ is $\rho_{f_{\eps,r}}$,
\begin{align} \label{Est_I2}
I_1 &= \int_{ \bt^{d}} | E_{\eps,r}[f_{\e, r}](x) - E_{\eps,r}[\mu^N_{\e , r}](x) |^2 \r_{f_{\e,r}}(x) \di x \\
& \leq \lVert \r_{f_{\e,r}} \rVert_{L^{\infty}( \bt^{d})} \lVert \eps^{-2} K *_x \chi_r *_x \chi_r *_x (\r_{f_{\eps,r}} - \r_{\mu^N_{\eps,r}}) \rVert^2_{L^2( \bt^{d})}.
\end{align}

Next we apply the following Loeper-type estimate; this is Lemma 3.2 of \cite{HKI15}:
\begin{lem} 
\label{Lem_HKI} Let $ h_1, h_2 \in L^{\infty}$ be probability densitites on $ \bt^{d}$. Let $\Psi_i$ satisfy
\be \label{Loep_electric}
\e^2 \Delta \Psi_i = h_i - 1 .
\ee
Then
\begin{equation} \label{Loep}
\lVert \nabla ( \Psi_1 - \Psi_2 ) \rVert_{L^{2}} \leq \eps^{-2} \max_i \lVert h_i \rVert_{L^{\infty}}^{1/2} \, W_2(h_1, h_2) .
\end{equation}
\end{lem}

Applying the above with $h_1 = \chi_r * \chi_r * \r_{f_{\eps,r}}$, $h_2 = \chi_r * \chi_r * \r_{\mu^N_{\eps,r}}$ gives
\begin{align}
I_1 & \leq \e^{-4} \lVert \r_{f_{\eps,r}} \rVert_{L^{\infty}} \max(\lVert \chi_r * \chi_r * \r_{f_{\eps,r}}\rVert_{L^{\infty}}, \lVert \chi_r * \chi_r * \r_{\mu^N_{\eps,r}} \rVert_{L^{\infty}}) \\
& \qquad \times W_2( \chi_r * \chi_r * \r_{f_{\e,r}}, \chi_r * \chi_r *\r_{\mu^N_{\eps,r}})^2.
\end{align}

Since the $L^{\infty}$ norm is not increased by convolution with a mollifier of mass one, the mollified densities can be controlled as follows:
\be
\max(\lVert \chi_r*\chi_r*\rho_{f_{\e,r}} \rVert_{L^{\infty}}, \lVert\chi_r*\chi_r*\rho_{\mu^N_{\e,r}} \rVert_{L^{\infty}} )\le \max(\lVert \rho_{f_{\e,r}} \rVert_{L^{\infty}}, \lVert \chi_r*\rho_{\mu^N_{\e,r}} \rVert_{L^{\infty}})
\ee

It remains to estimate $ \chi_r*\rho_{\mu^N_{\e,r}}$. The idea is to use that $\mu^N_{\e,r}$ should be close to $f_{\e,r}$, which has bounded mass density. For this we need the following estimate, which is Lemma 4.3 of \cite{Laz15} (see also \cite{BGV}, Prop 2.1). The proof is given there for the case of measures on $\br^d$, but adapts easily to the case of $\bt^d$.
\begin{lem} 
\label{Lem_mu_moll} Let $\nu$ be a probability measure on $\bt^d$ and $h \in L^{\infty}(\bt^d)$ a probability density function. Then, for all $r > 0$,
\begin{equation} \label{Est_mu_moll}
\lVert \chi_r * \nu \rVert_{L^{\infty}} \leq C \left ( \rVert h \lVert_{L^{\infty}} + r^{-(2 + d)} W_2^2(h, \nu) \right ).
\end{equation}
\end{lem}

We apply this with $h =  \r_{f_{\eps,r}} $, $\nu = \r_{\mu^N_{\eps,r}}$. Combining this with \eqref{Est_W_conv} and \eqref{Est_Wx_D}, we conclude that
\begin{align}
I_1 & \leq C \e^{-4} \lVert \r_{f_{\eps,r}}\rVert_{L^{\infty}} \left ( \rVert \r_{f_{\eps,r}} \lVert_{L^{\infty}} + r^{-(2+d)} W_2^2(\r_{f_{\eps,r}},\r_{\mu^N_{\eps,r}}) \right ) W_2^2(\r_{f_{\eps,r}},\r_{\mu^N_{\eps,r}}) \\
& \leq C \e^{-4} \lVert \r_{f_{\eps,r}} \rVert_{L^{\infty}} \left ( \rVert \r_{f_{\eps,r}} \lVert_{L^{\infty}} + r^{-(2+d)} \lambda^{-2} D(t) \right ) \lambda^{-2} D(t) \\
& \leq C \e^{-4} M \left ( M + r^{-(d+2)} \lambda^{-2} D(t) \right) \lambda^{-2} D(t) .
\end{align}

Next we estimate $I_2$. The regularisation procedure gives the force some Lipschitz regularity. This is quantified in the following result, which is an adaptation of Lemma 4.2(ii) of \cite{Laz15} to the case of the torus.
\begin{lem} \label{Lem_lip_moll}
Let $h \in L^{\infty}$. There exists $C > 0$ such that
\begin{equation} \label{Lip_moll}
\lVert \chi_r * K * h \rVert_{\text{Lip} } \leq C | \log{r} | (1 + \lVert h \rVert_{L^{\infty}}) .
\end{equation}
\end{lem}
\begin{proof}
First, using the representation of $K$ stated in Lemma~\ref{Lem_G}, we have 
\be
| \nabla (\chi_r * K * h) | \leq C |\nabla(\chi_r * \frac{x}{|x|^d} * h)| + C |\chi_r * \nabla K_0 * h| .
\ee

Since $K_0$ is smooth,
\be
\lVert \chi_r * \nabla K_0 * h \rVert_{L^{\infty}(\bt^d)} \leq \lVert \chi_r \rVert_{L^1(\bt^d)} \lVert\nabla K_0 \rVert_{L^1(\bt^d)} \lVert h \rVert_{L^{\infty}(\bt^d)} 
\ee

Let $\eta$ be a radially symmetric smooth bump function with $0 \leq \eta \leq 1$ such that
\be
\begin{array}{c} \label{eta_support}
\text{supp} \{\eta \}  \subset B_{2 l} \\
\text{supp} \{1 - \eta \} \cap B_{l} =  \emptyset ,
\end{array}
\ee
for some $l > 0 $ to be chosen later. We can take $\eta$ to satisfy 
\be \label{est_gradeta} |\nabla \eta| \leq C l^{-1}\ee
 for some $C > 0$. Then
  \begin{align} 
\left \lVert \nabla \left ( \chi_r * \frac{x}{|x|^d} * h \right ) \right \rVert_{L^{\infty}(\bt^d)} & \leq \left \lVert \nabla \[ \chi_r * \left ( \eta \frac{x}{|x|^d} \right  ) * h \] \right \rVert_{L^{\infty}(\bt^d)} + \left \lVert \nabla \[ \chi_r * \left ( (1 - \eta) \frac{x}{|x|^d} \right  ) * h \] \right \rVert_{L^{\infty}(\bt^d)} \\
& \leq \left \lVert (\nabla  \chi_r ) * \left ( \eta \frac{x}{|x|^d} \right  ) * h \right \rVert_{L^{\infty}(\bt^d)} +   \left \lVert \chi_r * \[ (1 - \eta) \nabla \left ( \frac{x}{|x|^{-d}} \right) \] * h \right \rVert_{L^{\infty}(\bt^d)} \\ \notag
& \qquad \qquad + \left \lVert \chi_r * \left ( |\nabla (1- \eta)| \frac{x}{|x|^{d}} \right ) * h \right \rVert_{L^{\infty}(\bt^d)} \\
& \leq \left \lVert \left (\nabla  \chi_r \right ) * \left ( \eta \frac{x}{|x|^d} \right  ) * h \right \rVert_{L^{\infty}(\bt^d)} + d \, \left \lVert \chi_r * \[ (1 - \eta) |x|^{-d} \] * h) \right \rVert_{L^{\infty}(\bt^d)} \\ \notag
& \qquad \qquad + \left \lVert \chi_r * \left( |\nabla \eta| |x|^{-(d-1)} \right ) * h \right \rVert_{L^{\infty}(\bt^d)} .
\end{align}
We estimate each term using Young's inequality: for any $p,q,r \in [1, \infty]$ satisfying $\frac{1}{p} + \frac{1}{q} = 1 + \frac{1}{r}$,
$$
\lVert f * g \rVert_{L^{r}} \leq \lVert f \rVert_{L^{p}} \lVert g \rVert_{L^{q}}. 
$$
For the first term, we apply Young's inequality twice to get
\begin{align}
\left \lVert (\nabla \chi_r ) * \left ( \eta \frac{x}{|x|^d} \right  ) * h \right \rVert_{L^{\infty}(\bt^d)} & \leq \left \lVert  \nabla \chi_r \right \rVert_{L^{1}(\bt^d)} \left \lVert \eta |x|^{-(d-1)} \right \rVert_{L^1(\bt^d)} \left \lVert h \right \rVert_{L^{\infty}(\bt^d)} \\
& \leq r^{-1} \left \lVert \nabla \chi  \right \rVert_{L^{1}(\bt^d)} \left \lVert |x|^{-(d-1)} \right \rVert_{L^1(B_{2 l })} \left \lVert h \right \rVert_{L^{\infty}(\bt^d)} \\
& \leq C \frac{l}{r} \left  \lVert h \right \rVert_{L^{\infty}(\bt^d)} .
\end{align}
The second line follows from scaling $\chi_r$ and the support condition on $\eta$. The third line follows from integrating $|x|^{-(d-1)}$ over the given set.

For the second term we have
\begin{align}
\lVert \chi_r * [ (1 - \eta) |x|^{-d} ] * h) \rVert_{L^{\infty}(\bt^d)} & \leq \lVert \chi_r \rVert_{L^{1}(\bt^d)} \lVert (1-\eta) |x|^{-d} \rVert_{L^1(\bt^d)} \lVert h \rVert_{L^{\infty}(\bt^d)} \\
& \leq \lVert |x|^{-d} \rVert_{L^1(\bt^d \setminus B_l)} \lVert h \rVert_{L^{\infty}(\bt^d)} \\
& \leq C (1 - \log{l})  \lVert h \rVert_{L^{\infty}(\bt^d)}
\end{align}
Similarly to the previous estimate, the second line follows from $\chi_r$ having unit mass and from the support condition on $1 -\eta$. The third line follows from integrating $|x|^{-d}$ over the given set.

For the third term we have
\begin{align}
\lVert \chi_r * [ | \nabla \eta | |x|^{-(d-1)} ] * h \rVert_{L^{\infty}(\bt^d)} & \leq \lVert \chi_r \rVert_{L^{1}(\bt^d)} \lVert | \nabla \eta | |x|^{-(d-1)} \rVert_{L^1(\bt^d)} \lVert h \rVert_{L^{\infty}(\bt^d)} \\
& \leq \lVert |\nabla \eta| \rVert_{L^{\infty}(\bt^d)} \lVert |x|^{-(d-1)} \rVert_{L^1(B_{2l} \setminus B_l)} \lVert h \rVert_{L^{\infty}(\bt^d)} \\
& \leq C l^{-1} \lVert |x|^{-(d-1)} \rVert_{L^1(B_{2l} \setminus B_l)} \lVert h \rVert_{L^{\infty}(\bt^d)} \\
& \leq C  \lVert h \rVert_{L^{\infty}(\bt^d)}
\end{align}
The second line follows from H\"{o}lder's inequality, the fact that $\chi_r$ has unit mass, and the fact that $\eta$ is constant outside of $B_{2l} \setminus B_l$. For the third line we use the bound \eqref{est_gradeta} that we assumed on $\nabla \eta$. The fourth line follows by integrating $|x|^{-d}$ over the relevant set.

Thus
\be
\left \lVert \nabla \left (\chi_r * K * h \right) \right \rVert_{L^{\infty}(\bt^d)} \leq C\left(1 + l r^{-1} - \log{l} \right)  \lVert h \rVert_{L^{\infty}(\bt^d)} .
\ee

Choosing $l = r$ gives the desired result.

\end{proof}

To estimate the term $I_2$ appearing in \eqref{Est_Dprime_1} we apply the above result with $h = \chi_r * \r_{\mu^N_{\eps,r}}$ and we obtain:
\begin{align}
I_2 &= \int_{ (\bt^{d} \times \br^{d})^2} | E_{\eps,r}[\mu^N_{\e,r}](x) - E_{\eps,r}[\mu^N_{\e,r}](y) |^2 \di \pi \\
&= \int_{ (\bt^{d} \times \br^{d})^2} | \chi_r * \eps^{-2} K * (\chi_r * \r_{\mu^N_{\eps, r}}) (x) - \chi_r * \eps^{-2} K * (\chi_r * \r_{\mu^N_{\eps, r}}) (y) |^2 \di \pi \\
& \leq \int_{ (\bt^{d} \times \br^{d})^2} \eps^{-4} \lVert \chi_r *  K * (\chi_r * \r_{\mu^N_{\eps, r}}) \rVert^2_{\text{Lip}} \, |x-y|^2 \di \pi \\
& \leq C \eps^{-4} |\log{r}|^2 (1 + \lVert \chi_r * \r_{\mu^N_{\eps,r}} \rVert_{L^{\infty}})^2 \int_{ (\bt^{d} \times \br^{d})^2} |x-y|^2 \di \pi
\end{align}

To this we apply Lemma \ref{Lem_mu_moll} and \eqref{Est_Wx_D}, which gives
\begin{align}
I_2 & \leq C \eps^{-4} |\log{r}|^2 \left (1 + \lVert \r_{f_{\eps,r}} \rVert_{L^{\infty}} + r^{-(d+2)} W_2^2(\r_{\mu^N_{\eps,r}}, \r_{f_{\eps, r}}) \right )^2 \int_{ (\bt^{d} \times \br^{d})^2} |x-y|^2 \di \pi \\
& \leq C \eps^{-4} |\log{r}|^2 \left (1 + \lVert \r_{f_{\eps,r}} \rVert_{L^{\infty}} + r^{-(d+2)} \lambda^{-2} D \right )^2 \lambda^{-2} D \\
& \leq C \eps^{-4} |\log{r}|^2 \left (M + r^{-(d+2)} \lambda^{-2} D \right )^2 \lambda^{-2} D,
\end{align}
where $M$ is as in \eqref{Def_M} and \eqref{M_z}.

Altogether this results in the differential inequality
\begin{equation} \label{Est_Dprime_2}
D' \leq (\lambda + \alpha)D + \frac{C \eps^{-4}}{\alpha} \left ( | \log{r} |^2 \left (M + r^{-(2+d)} \lambda^{-2} D \right )^2  + M \left(M + r^{-(2+d)} \lambda^{-2} D \right) \right ) \lambda^{-2} D .
\end{equation}

\subsection{Closing the estimate and choice of parameters}

Following \cite{Laz15}, we consider the regime $D \leq \lambda^2 r^{2+d}$ to remove the effect of the non-linear part of the estimate. As long as this condition holds, the estimate becomes
\begin{equation} \label{Est_Dprime_3}
D' \leq (\lambda + \alpha)D + \frac{C \eps^{-4}}{\alpha} M^2 |\log{r}|^2 \lambda^{-2} D .
\end{equation}
Optimising over $\alpha$ leads us to choose $\alpha = C \lambda^{-1} \eps^{-2} |\log{r}| M$. With this choice \eqref{Est_Dprime_3} becomes
\begin{equation} \label{Est_Dprime_OptA}
D' \leq (\lambda + C \eps^{-2} |\log{r}| M \lambda^{-1})D .
\end{equation}
The optimal value of $\lambda$ is $C \eps^{-1} |\log{r}|^{1/2} M^{1/2}$, giving
\begin{equation} \label{Est_Dprime_OptL}
D' \leq C \eps^{-1} |\log{r}|^{1/2} M^{1/2} D .
\end{equation}

Thus
\begin{equation} \label{Est_Gron_Full}
D(t) \leq D(0) e^{C \eps^{-1} |\log{r}|^{1/2} M^{1/2} t} ,
\end{equation}
provided that $D$ never exceeds $\lambda^2 r^{d+2}$. 

\subsection{Modified distance}

In order to unify \eqref{Est_Gron_Full} into a single statement, we define a new distance $\hat{D}$ by truncating $D$ at the level $\lambda^2 r^{d+2}$. Using this \eqref{Est_Gron_Full} can then be rephrased as a straightforward Gronwall estimate for the new distance with no extra conditions on the magnitude of $\hat{D}$. We also rescale the distance so as to work with quantities of order 1.
\be
\hat{D}(t) = \min \{1 , \lambda^{-2} r^{-(d+2)} \sup_{s \in [0,t]} D(s) \} .
\ee
Then \eqref{Est_Gron_Full} implies that
\be
\hat{D}(t) \leq \hat{D}(0) e^{C \eps^{-1} |\log{r}|^{1/2} M^{1/2} t} .
\ee

When $\hat{D}$ is smaller than the truncation level, it can be used to control the Wasserstein distance. That is, if there exists some initial coupling $\pi_0$ such that $\hat{D}(t) < 1$, then by \eqref{ineq_W-D}
\be \label{Est_W-Dhat}
\sup_{[0,t]} W_2^2(f_{\e, r}, \mu^N_{\e,r}) \leq \lambda^2 r^{d+2} \inf_{\pi_0} \hat{D} (t) .
\ee

\subsection{Convergence rates for $r$ and $\e$}

Our next goal is to show that we can choose $\e = \e_N$, $r = r_N$ so that for sufficiently well-placed initial configurations we have $\hat{D}(T) \to 0$ as $N \to \infty$. By the previous discussion \eqref{Est_W-Dhat} this will imply convergence for the Wasserstein distance. The key point is to choose $\e$ and $r$ so as to prevent the exponential growth factor $e^{C \eps^{-1} |\log{r}|^{1/2} M^{1/2} t}$ from exploding too quickly with $N$; then the convergence will hold as long as $\hat{D}(0)$ converges to 0 sufficiently quickly.

\begin{lem} \label{Lem_valid}
Let $f_{0, \e}$ be initial data satisfying Assumption~\ref{Hyp_data}. Recall that we have defined the parameter $\zeta$ associated with these data in \eqref{Def_z:2d}-\eqref{Def_z:3d}.
There exists $C$ such that the following holds: let $\beta > 0$, and let $\eps = \eps_N$, $r = r_N$ be chosen to satisfy
\be \label{Gron_r_eps}
r < e^{- C \eps^{-2 - d \zeta} \, \frac{  T^2}{\beta^2}} .
\ee
Assume that for some $\eta > \beta$ the initial configurations $\mu^N_{\e}(0)$ satisfy
\be \label{Hyp_init_conv_2}
\lim_{N \rightarrow \infty} \frac{W_2(\mu^N_{\eps}(0), f_{0, \eps})}{\eps^{- \gamma} r^{1 + d/2 + \eta/2}} = 0 .
\ee
Then as $N \to \infty$, $\inf_{\pi_0} \hat{D} (T) \to 0$ .

\end{lem}

\begin{proof}
First we check that $\inf_{\pi_0} \hat{D}(0) \to 0$. For convenience we will define a function $\omega$ by
\be \label{Def_omega}
\omega(N) = \frac{W_2(\mu^N_{\eps}(0), f_{0, \eps})}{\eps^{- \gamma} r^{1 + d/2 + \eta/2}} \, ;
\ee
thus \eqref{Hyp_init_conv_2} implies that $\omega$ is bounded. Recalling \eqref{Est_infD-W}, 
\be
\inf_{\pi_0} D(0) \leq \frac{1}{2} \lambda^2 W^2_2(\mu_{\eps}^N(0), f_{0, \eps}) .
\ee
Thus, since by definition $\hat{D}(0) \leq  \lambda^{-2} r^{-(d+2)} D(0)$, it follows that
\be
\inf_{\pi_0} \hat{D}(0) \leq \frac{1}{2} r^{-(d+2)} W^2_2(\mu_{\eps}^N(0), f_{0, \eps}) .
\ee
By definition of $\omega$ we have
\be
\inf_{\pi_0} \hat{D}(0) \leq \frac{1}{2} r^{\eta} \e^{- 2 \gamma} \omega(N)^2 .
\ee
If \eqref{Gron_r_eps} holds, then $\e^{-1} \leq C_{\beta, T} |\log{r}|^{\frac{1}{2 + d \zeta}}$ and thus
\be
\inf_{\pi_0} \hat{D}(0) \leq C r^{\eta} |\log{r}|^{\frac{2 \gamma}{2 + d \zeta}} \omega(N)^2 .
\ee
Then for any $\eta ' < \eta$ there exists $C = C_{\beta, T, \eta ', \gamma, \zeta}$ such that
\be
\inf_{\pi_0} \hat{D}(0) \leq C r^{\eta '}.
\ee

Next, we use the Gronwall estimate to get convergence at later times. We can do this by controlling the exponential growth factor in \eqref{Est_Gron_Full}. Observe that for $r < 1$ this factor satisfies
\be \label{growthfactor}
e^{C M^{1/2} \eps^{-1} |\log{r}|^{1/2} T} = r^{- C M^{1/2} \eps^{-1} |\log{r}|^{- 1/2} T} .
\ee
If $\e, r$ satisfy \eqref{Gron_r_eps}, then by \eqref{M_z}
\be \label{Est_exp}
C M^{1/2} \eps^{-1} |\log{r}|^{- 1/2} T \leq C T \eps^{-1 - \frac{d \zeta}{2}} |\log{r}|^{- 1/2} \leq \beta
\ee
and hence
\be \label{Est_growth}
r^{- C M^{1/2} \eps^{-1} |\log{r}|^{- 1/2} T} \leq r^{- \beta} .
\ee
This implies that for all $t \in [0,T]$
\begin{align}
\inf_{\pi_0} \hat{D}(t) & \leq \inf_{\pi_0} \, \hat{D}(0) \, r^{- \beta} \\
& \leq C r^{\eta ' - \beta} .  \label{hatD_ctrl}
\end{align}
We complete the proof by choosing $\eta ' > \beta$.

\end{proof}

We can now complete the proof of Proposition \ref{Prop_MFL}:

\begin{proof}[End of proof of Proposition \ref{Prop_MFL}] By Lemma \ref{Lem_valid}, under condition \eqref{Hyp_init_conv} we have $\inf_{\pi_0} \hat{D} (T) \to 0$ as $N \to \infty$. In particular, for $N$ sufficiently large we have $\inf_{\pi_0} \hat{D}(T) < 1$. Then as previously discussed in \eqref{Est_W-Dhat},
\be \label{W_D}
\sup_{[0,T]} W_2^2(f_{\e,r}, \mu^N_{\e,r}) \leq \lambda^2 r^{d+2} \inf_{\pi_0} \hat{D}(T) .
\ee

Recall that we chose $\lambda = C M^{1/2} \eps^{-1} |\log{r}|^{1/2} $. Thus by \eqref{Est_exp} we have $\lambda \leq C \beta T^{-1} |\log{r}|$. Hence 
\be \label{Est-lr}
\lambda^2 r^{d+2} \leq C_{\beta, T}  |\log{r}|^2 r^{d+2} .
\ee

As in the proof of Lemma~\ref{Lem_valid}, we use that \eqref{r_eps_2} implies \eqref{hatD_ctrl}:
\be \label{hatD_ctrl_rep}
\inf_{\pi_0} \hat{D}(T) \leq C r^{\eta ' - \beta} ,
\ee
for any $\beta < \eta ' < \eta$ and some $C = C_{\beta, T, \eta ' , \gamma, \zeta}$. Then by combining \eqref{Est-lr} and \eqref{hatD_ctrl_rep} with \eqref{W_D} we obtain
\be
\sup_{[0,T]} W_2^2(f_{\e, r}, \mu^N_{\e,r}) \leq C | \log{r}|^2 r^{d+2 +\eta ' - \beta} .
\ee

By adjusting $\eta '$ and $C$ so as to absorb the logarithmic factor, we may conclude that for $N$ sufficiently large,
\be
\sup_{[0,T]} W_2^2(f_{\e, r}, \mu^N_{\e,r}) \leq C r^{d+2 +\eta ' - \beta} ,
\ee
as desired.
\end{proof}

\subsection{Regularised and Unregularised Vlasov-Poisson}

In this section, we prove the following Gronwall-type estimate between solutions of the regularised and unregularised Vlasov-Poisson systems.

\begin{prop} \label{Prop_RUR} \begin{enumerate}[(i)]
\item Let $f_{\eps,r}$ be a solution of \eqref{Eqn_VP_reg} and $f_{\eps}$ a solution of \eqref{Eqn_VP_eps}, both having the same initial datum $f_{0, \eps}$. Let $M = M_{\e, T}$ be chosen such that \eqref{Def_M} is satisfied. 
Then there exists a constant $C$, independent of $r$, $M$ and $\eps$, such that for all $t \in [0,T]$
\be \label{Est_RUR}
W_2(f_{\eps, r}(t), f_{\eps}(t)) \leq C \eps^{-3/2} M^{3/4} r | \log{r} |^{-1/4} e^{C \eps^{-1} M^{1/2} | \log{r} |^{1/2} t} .
\ee
\item Let $\{f_{0, \e}\}$ be a set of initial data satisfying Assumption~\ref{Hyp_data}. Let $\zeta$ be defined as in \eqref{Def_z:2d}-\eqref{Def_z:3d}, and let $T > 0$ be fixed. If $\e = \e_N$ and $r = r_N$ are chosen to satisfy \eqref{r_eps_2} for some $\beta < 1$, then
\be \label{RUR_Wconv}
\lim_{N \to \infty} \sup_{t \in [0,T]} W_2(f_{\eps, r}(t), f_{\eps}(t)) = 0.
\ee
\end{enumerate}
\end{prop}
\begin{proof}
Once again we consider couplings of $f_{\eps,r}$ and $f_{\eps}$ that evolve according to the dynamics of their respective equations. Since $f_{\eps, r}$ and $f_{\eps}$ have the same initial datum $f_{0, \eps}$, the choice $\pi_0(x,v,y,w) = f_0(x,v) \delta(x-y, v-w)$ is an optimal initial coupling, so that at time $t=0$ we have
\be \label{RUR_initial}
W_2^2(f_{\eps, r}, f_{\eps}) =  \int_{ (\bt^{d} \times \br^{d})^2 } |x-y|^2 + |v-w|^2 \di \pi_0(x,v,y,w)
\ee

We take $\pi_t$ to solve
$$
(\pt_t  + v \cdot \nabla_x + w \cdot \nabla_y) \pi_t + (E_{\e,r}[f_{\e,r}](x) \cdot \nabla_v + E_\e[f_\e](y) \cdot \nabla_w) \pi_t = 0 .
$$
Then $\pi_t$ is a coupling of $f_{\eps, r}(t)$ and $f_{\eps}(t)$ for all times $t$. We want to estimate
\be \label{RUR_Wt}
\int_{ (\bt^{d} \times \br^{d})^2 } |x-y|^2 + |v-w|^2 \di \pi_t(x,v,y,w) ,
\ee
which controls $W_2^2(f_{\eps, r}, f_{\eps})$ by definition. To this end we define an anisotropic distance $D$ as before: let
$$
D(t) = \frac{1}{2} \int_{ (\bt^{d} \times \br^{d})^2 } \lambda^2 |x-y|^2 + |v-w|^2 \di \pi_t(x,v,y,w)
$$
for some $\lambda$ to be specified later. Then, as in \eqref{Est_Dprime_1}, we estimate that
\begin{align} \label{RUR_Dprime}
D'(t) & \leq (\lambda + \alpha) D(t) +  \frac{1}{2 \alpha} \int_{ (\bt^{d} \times \br^{d})^2} | E_{\eps,r}[f_{\eps,r}](x) - E_{\eps}[f_{\eps}](y) |^2 \di \pi_t .
\end{align}

By the triangle inequality, the second term in \eqref{RUR_Dprime} may be estimated as follows:
\begin{multline} \label{RUR_triangle}
\frac{1}{2 \alpha} \int_{ (\bt^{d} \times \br^{d})^2} | E_{\eps,r}[f_{\eps,r}](x) - E_{\eps}[f_{\eps}](y) |^2 \di \pi_t \leq \frac{1}{\alpha} \underbrace{  \int_{ (\bt^{d} \times \br^{d})^2} | E_{\eps,r}[f_{\eps,r}](x) - E_{\eps,r}[f_{\eps,r}](y) |^2 \di \pi_t }_{: = I_1} \\ 
+ \frac{1}{\alpha} \underbrace{ \int_{ (\bt^{d} \times \br^{d})^2} | E_{\eps,r}[f_{\eps,r}](y) - E_{\eps}[f_{\eps}](y) |^2 \di \pi_t }_{: = I_2}. 
\end{multline}

The term $I_1$ is estimated using the Lipschitz regularity of the regularised forces (Lemma \ref{Lem_lip_moll}):

\begin{align}
I_1 &= \int_{ (\bt^{d} \times \br^{d})^2} | E_{\eps,r}[f_{\eps,r}](x) - E_{\eps,r}[f_{\eps, r}](y) |^2 \di \pi \\
& \leq \int_{ (\bt^{d} \times \br^{d})^2} \lVert \eps^{-2} \chi_r *  K * (\chi_r * \r_{f_{\eps,r}} ) \rVert^2_{\text{Lip}} \, |x-y|^2 \di \pi \\
& \leq C \eps^{-4} |\log{r}|^2 (1 + \lVert \chi_r * \r_{f_{\eps,r}} \rVert_{L^{\infty}})^2 \int_{ (\bt^{d} \times \br^{d})^2} |x-y|^2 \di \pi \\
& \leq C \eps^{-4} | \log{r} |^2 M^2 \lambda^{-2} D(t) .
\end{align}

For $I_2$, we first observe that since the $y$-marginal of $\pi_t$ is $\r_{f_{\eps}}(y) \di y$,
\begin{align} \label{RUR_I2def}
I_2 &= \int_{ \bt^{d}} | \eps^{-2} K * (\chi_r * \chi_r * \r_{f_{\eps,r}} - \r_{f_{\eps}}) |^2(y) \r_{f_{\eps}}(y) \di y \\
& \leq M \lVert \eps^{-2} K * (\chi_r * \chi_r * \r_{ f_{\eps,r}} - \r_{f_{\eps}}) \rVert^2_{L^2(\bt^d)}, \label{I2_ineq}
\end{align}
where \eqref{I2_ineq} follows from the fact that $\lVert \r_{f_{\e}} \rVert_{L^{\infty}} \leq M$ by definition of $M$ (recall \eqref{Def_M}).

We then apply the Loeper-type estimate in Lemma \ref{Lem_HKI}:
\be
\lVert \eps^{-2} K * (\chi_r * \chi_r * \r_{f_{\eps,r}} - \r_{f_{\eps}}) \rVert_{L^2(\bt^d)} \leq \eps^{-2} M^{1/2} W_2(\chi_r * \chi_r * \r_{ f_{\eps,r}}, \r_{f_{\eps}}) .
\ee

To control the Wasserstein distance we first apply the triangle inequality:
\be \label{RUR_W_triangle}
W_2(\chi_r * \chi_r * \r_{ f_{\eps,r}}, \r_{ f_{\eps}}) \leq W_2(\chi_r * \chi_r * \r_{ f_{\eps,r}}, \chi_r * \r_{f_{\eps, r}}) + W_2(\chi_r * \r_{f_{\eps,r}}, \r_{ f_{\eps,r}}) + W_2(\r_{f_{\eps,r}}, \r_{ f_{\eps}}) .
\ee

The third term can be controlled by $\lambda^{-1}D^{1/2}$ due to \eqref{Est_Wx_D}. We apply Lemma \ref{W_moll_one} to each of the first two terms. This results in the estimate
\be \label{RUR_W_est}
W_2(\chi_r * \chi_r * \rho_{f_{\eps,r}}, \rho_{f_{\eps}}) \leq 2r + \lambda^{-1} D^{1/2} .
\ee
Thus we get the following estimate for $I_2$:
\begin{align} \label{RUR_I2_final}
I_2 & \leq C \eps^{-4} M^2 (2r + \lambda^{-1} D^{1/2})^2 \\
& \leq C \eps^{-4} M^2 (r^2 + \lambda^{-2}D) .
\end{align}

Substituting these estimates into \eqref{RUR_Dprime} gives us that
\begin{align} \notag
D'(t) & \leq \left [ \lambda + \alpha + \frac{1}{\alpha} C \eps^{-4} M^2(|\log{r}|^2  + 1) \lambda^{-2} \right ]D(t) + \frac{C}{\alpha} \eps^{-4} M^2 r^2 \\ \label{RUR_Dprime_wparam}
& \leq \left [ \lambda + \alpha + \frac{1}{\alpha} C \eps^{-4} M^2 |\log{r}|^2 \lambda^{-2} \right ]D(t) + \frac{C}{\alpha} \eps^{-4} M^2 r^2 .
\end{align}

We will now choose the parameters $\alpha$ and $\lambda$ so as to minimise the constant in the exponential part of our Gronwall estimate; that is, the coefficient of $D$ in \eqref{RUR_Dprime_wparam}. This has a minimum for $(\alpha, \lambda)$ satisfying
\be
\left \{ \begin{array}{ccc}
\alpha &=& C \lambda^{-1} \eps^{-2} M | \log{r} | \\
\lambda &=& \left ( \frac{C \eps^{-4} M^2 | \log{r} |^2}{\alpha} \right )^{1/3} ,
\end{array} \right.
\ee
that is, $\alpha = \lambda = C \eps^{-1} M^{1/2} | \log{r} |^{1/2}$. With this choice of parameters \eqref{RUR_Dprime_wparam} becomes
\be   \label{RUR_Dprime_final}
D'(t) \leq C \eps^{-1} M^{1/2} | \log{r} |^{1/2} D(t) + C \eps^{-3} M^{3/2} r^2 | \log{r} |^{-1/2} .
\ee

Since $D(0) = 0$, the above inequality implies that
\be \label{RUR_D_Gron}
D(t) \leq C \eps^{-3} M^{3/2} r^2 | \log{r} |^{-1/2} e^{C \eps^{-1} M^{1/2} | \log{r} |^{1/2} t} .
\ee

We conclude that
\be \label{RUR_W_Gron}
W_2(f_{\eps, r}(t), f_{\eps}(t)) \leq C \eps^{-3/2} M^{3/4} r | \log{r} |^{-1/4} e^{C \eps^{-1} M^{1/2} | \log{r} |^{1/2} t} .
\ee

Finally, we use this to prove \eqref{RUR_Wconv}. Since $\{f_{0, \e}\}$ are assumed to satisfy Assumption~\ref{Hyp_data}, we may apply Propositions \ref{prop_dens_2d} and \ref{prop_dens_3d} to deduce that
\be
M \leq C \e^{- \zeta d},
\ee
for $\zeta$ defined in \eqref{Def_z:2d}-\eqref{Def_z:3d}. Next, as in the proof of Lemma~\ref{Lem_valid}, we observe that the relation \eqref{r_eps_2} implies that
\be
\e^{-(2 + \zeta d)} \leq C_{T, \beta} |\log{r}|,
\ee
for some $C_{\beta, T}$. Thus
\be
\e^{-1} M^{1/2} \leq C_{\beta, T} |\log{r}|^{1/2} .
\ee

Moreover, by \eqref{growthfactor} and \eqref{Est_growth}, we have
\be
e^{C \eps^{-1} M^{1/2} | \log{r} |^{1/2} t} \leq r^{- \beta} .
\ee

Thus
\be
\sup_{t \in [0,T]} W_2(f_{\eps, r}(t), f_{\eps}(t)) \leq C_T \, r^{1 - \beta} |\log{r}|^{1/2} .
\ee

Since $\beta <1$, the right hand side converges to 0 as $N \to \infty$. This completes the proof.

\end{proof}

\section{Quasi-neutral limit}

Next we perform the quasineutral limit on the mean field equation, i.e. the limit \eqref{Eqn_VP_eps} $\to$ \eqref{Eqn_KIE}. This is the content of Theorem 1 of \cite{HKI15}, recalled below.

\begin{thm} 
\label{multiD:thm1}
Let $\gamma$, $\delta_0$,
and $C_0$ be positive constants, with $\delta_0>1$.
Consider a sequence $(f_{0,\eps})$ of non-negative initial data in $L^1$ satisfying Assumption~\ref{Hyp_data}. For all $\e \in (0,1)$, consider $f_\e(t)$  a global weak solution of $(VP)_\eps$ with initial condition $f_{0,\e}$. Define the filtered distribution function
\begin{equation} 
\widetilde{f}_\e(t,x,v) := f_\e \Big(t,x,v-\frac{1}{i}(d_+(t,x)e^{\frac{it}{\sqrt \e}}-d_-(t,x)e^{-\frac{it}{\sqrt \e}})\Big)
\end{equation}
where $(d_\pm)$ are defined in \eqref{Def_d}-\eqref{Def_d_3}.

There exist $T>0$ and $g(t)$ a weak solution on $[0,T]$ of (KIE) with initial condition $g_0$ such that
$$
\lim_{\e \to 0} \sup_{t \in [0,T]}  W_1(\widetilde{f}_\e(t), g(t)) = 0.
$$

\end{thm}

\section{Filtering and triangular argument}

Finally, we combine the previous results to complete the proof of Theorem \ref{thm_main}. We use $\tilde{\mu}$ to denote the distribution produced by filtering $\mu$ using the correctors defined in \eqref{corrector}, following Definition~\ref{Def_filt}.

\begin{proof}[Proof of Theorem \ref{thm_main}]

First we apply the triangle inequality for the Wasserstein distance to get
\be \label{N_KIE_triangle}
W_1(\tilde{\mu}^N_{\eps,r}, g) \leq W_1(\tilde{\mu}^N_{\eps,r}, \tilde{f}_{\eps}) + W_1({\tilde{f}_{\eps}, g})
\ee

We begin by estimating $W_1(\tilde{\mu}^N_{\eps,r}, \tilde{f}_{\eps})$. By Lemma \ref{Lem_filt} and the fact that $W_1 \leq W_2$,
\begin{align}
\sup_{[0,T]} W_1(\tilde{\mu}^N_{\eps,r}, \tilde{f}_{\eps}) & \leq C_T \, \sup_{[0,T]} W_1(\mu^N_{\eps,r}, f_{\eps}) \\
& \leq C_T \, \sup_{[0,T]} W_2(\mu^N_{\eps,r}, f_{\eps}) \\
& \leq C_T \, \left ( \sup_{[0,T]} W_2(\mu^N_{\eps,r}, f_{\eps, r} ) + \sup_{[0,T]} W_2( f_{\eps, r}, f_{\eps}) \right ) .
\end{align}

Under conditions \eqref{Hyp_init_conv} and \eqref{r_eps_2}, the first term converges to 0 by Proposition~\ref{Prop_MFL}. The second term converges by Proposition~\ref{Prop_RUR}. Hence
\be
\lim_{N \to \infty} \sup_{[0,T]} W_1(\tilde{\mu}^N_{\eps,r}, \tilde{f}_{\eps}) = 0 .
\ee

For $W_1({\tilde{f}_{\eps}, g})$, we apply Theorem \ref{multiD:thm1} to get that
\be
\lim_{N \to \infty} \sup_{[0,T]} W_1({\tilde{f}_{\eps}, g}) = 0 .
\ee

Therefore
\be
\lim_{N \to \infty} \sup_{[0,T]} W_1(\tilde{\mu}^N_{\eps,r}, g) = 0.
\ee

\end{proof}

\section{Convergence of initial data} \label{sec:conc}

It remains for us to show that conditions \eqref{init_conv_1d} and \eqref{Hyp_init_conv:thm} are reasonable ones to impose on the initial configurations. We would like to exhibit a large class of configurations for which these conditions hold. In this section we show that for certain ranges of the parameters $\e$ and $r$, these conditions are in fact `typical’ - meaning that the condition holds with high probability when the initial configurations $[(x_i(0), v_i(0))]_{i=1}^N$ are chosen by taking $N$ independent samples from the initial distribution $f_{0, \e}$. This can be thought of as considering the behaviour of $N$ typical particles. In the rest of this section, we will consider only the case where the initial configurations are chosen in this way. Thus $\mu^N_{\e}(0)$ will always denote the initial empirical measure constructed with this method. For this choice of initial configurations we will prove that \eqref{Hyp_init_conv} does indeed hold with probability 1, provided that $\e$ and $r$ converge at a suitable rate with respect to $N$. We give this rate explicitly.

The fact that an empirical measure constructed by taking independent samples from a given probability distribution approximates that same distribution as the number of samples tends to infinity is a well-known result in statistics. The main idea of this section is to use a concentration inequality to quantify the rate at which this convergence occurs in the Wasserstein. The concentration inequalities we use are due to Fournier-Guillin \cite{FG}. This approach was used in \cite{Laz15}; however in our case we use a slightly different version of the concentration inequalities in order to exploit our choice of compactly supported data with a controlled growth rate in Assumptions \ref{Hyp_data_1d} and \ref{Hyp_data}. 

\begin{prop}[One dimension] \label{Prop_typ_1d}
Let $d=1$. Let $f_{0,\e}$ satisfy Assumption~\ref{Hyp_data_1d} and the same support assumption \ref{Hyp_data_1d}(ii) as $g_{0,\e}$, that is, there exists $R > 0$ such that $f_{0,\e} = 0$ for all $|v| > R$ and all $x, \e$. Fix a constant $C_2 > 0$ and suppose that $\e = \e_N$ satisfies
\be \label{epsrate:1d}
\e \geq \frac{A}{\log{N}}
\ee
for some $A > 2 C_2$. Let $((x_i^{(\e)}, v_i^{(\e)}))^N_{i=1}$ be chosen by taking $N$ independent samples from $f_{0,\e}$. Let $\mu^N_\e(0)$ denote the associated empirical measure. Then with probability 1
\be
\lim_{\e \to 0} \e^{-1} e^{\e^{-1} C_2} W_1(\mu^N_\e(0), f_{0,\e}) = 0 .
\ee
\end{prop}

\begin{prop} \label{Prop_typ_high}
Let $d=2,3$. Let $f_{0,\e}$ satisfy Assumption~\ref{Hyp_data}. Let $r = r_N$ be chosen such that $r_N \geq N^{- \frac{1}{d(d+2)} + \alpha}$ for some $\alpha > 0$. Then there exists a constants $\eta > 0$, $C > 0$ such that with probability 1, for all $N$ sufficiently large,
\be
\frac{W_2^2(\mu^N_{\eps}(0), f_{0, \eps})}{\e^{- 2 \gamma} r^{d+2+\eta}} \leq C .
\ee
\end{prop}

Theorems \ref{thm_typ_1d} and \ref{thm_typ} follow immediately from the above propositions combined with Theorems \ref{thm_main_1d} and \ref{thm_main}.

\subsection{Concentration inequality}

The central result we will make use of is a concentration inequality for compactly supported measures in Wasserstein sense from Fournier-Guillin \cite{FG}.

\begin{thm} 
 \label{thm_FG} Let $\nu$ be a measure supported on $[-1, 1]^m$, and let $\nu^N$ denote the empirical measure of $N$ independent samples from $\nu$. Then there exist constants $c$, $C$ and $\kappa$ depending on $m$ and $p$ only such that for all $x > 0$,
\be \label{conc_cpct}
\PP \left( W_p^p (\nu, \nu^N) \geq \kappa x \right ) \leq C \mathbbm{1}_{\{ x \leq 1 \}} \left \{ \begin{array}{cc}
\exp{(-cN x^2)} & p > m/2 \\
\exp{(-cN (\frac{x}{\log{(2 + 1/x)}})^2)} & p = m/2 \\
\exp{(-cN x^{m/p})} & p \in (0, m/2) .
\end{array} \right.
\ee
\end{thm}

We can use this result to derive a rate of convergence by choosing $x$ dependent on $N$. However, while we have assumed that the initial measures we work with are compactly supported, under Assumptions \ref{Hyp_data_1d} and \ref{Hyp_data} the support may be larger than $[-1,1]^{2d}$. Thus the result above does not immediately cover our case. In the next section we use a scaling argument to derive a concentration inequality for measures with varying support.

\subsection{Scaling} \label{scaling}

For our purposes, we want to consider $\nu = f_{0,\eps} \di x \di v$, which has support contained in $\bt^d \times B_{R}$ for some $R>0$. In order to apply the above estimate, we first rescale the velocity variable in order to work with measures supported in $[-1,1]^{2d}$.

\begin{defi} \label{Def_scale}
Let $\nu$ be a measure on $\bt^d \times \br^d$. We define a scaled measure $\s_R [\nu]$ such that for any $X \in \mc{B}(\bt^d)$ and $V \in \mc{B}(\br^d)$,
\be
\s_R[\nu] (X \times V) = \nu(X \times R V) .
\ee
Similarly, let $\nu_1$ and $\nu_2$ be measures on $\bt^d \times \br^d$ and let $\pi \in \Pi(\nu_1, \nu_2)$. Then let $\s_R^{(2)}[\pi]$ be defined via
\be
\s_R^{(2)}[\pi] (X_1 \times V_1 \times X_2 \times V_2) = \pi(X_1 \times R V_1 \times X_2 \times R V_2) .
\ee
\end{defi}
\begin{rks}
\begin{enumerate}[(i)]
\item Note that $\s_R^{(2)}[\pi] \in \Pi(\s_R[\nu_1], \s_R[\nu_2])$ .
\item $\s_R^{(2)}$ gives a bijection between $\Pi(\nu_1, \nu_2)$ and $\Pi(\s_R[\nu_1], \s_R[\nu_2])$.
\end{enumerate}
\end{rks}

We examine the effect of this scaling on the Wasserstein distance.
\begin{lem} \label{lem_scale}
Let $\nu_1$, $\nu_2$ be measures on $\bt^d \times \br^d$. Then
\be \label{Est_W_scale}
W_p(\nu_1, \nu_2) \leq R W_p(\s_R[\nu_1], \s_R[\nu_2]) . 
\ee
\end{lem}
\begin{proof}
Observe that for any $\pi \in \Pi(\nu_1, \nu_2)$,
\begin{align}
\int_{(\bt^d \times \br^d)^2} |x-y|^p + |v-w|^p \di \pi & = \int_{(\bt^d \times \br^d)^2} |x-y|^p + R^p |v-w|^p \, \di \s_R^{(2)}[\pi] \\
& \leq R^p \int_{(\bt^d \times \br^d)^2} |x-y|^p + |v-w|^p \, \di \s_R^{(2)}[\pi] 
\end{align}

Since $\s_R^{(2)}$ is a bijection, taking infimum over $\pi$ yields
\be
W_p(\nu_1, \nu_2) \leq R W_p(\s_R[\nu_1], \s_R[\nu_2]) . 
\ee

\end{proof}

\subsection{Typicality in one dimension}

\begin{proof}[Proof of Proposition~\ref{Prop_typ_1d}]

We will show that there exists $\alpha > 0$ such that with probability 1
\be
\e^{-1} e^{\e^{-1} C_2} W_1(\mu^N_\e(0), f_{0,\e}) \leq C N^{- \alpha}
\ee
for some $C$ and all sufficiently large $N$. This will follow from the Borel-Cantelli lemma provided that we can show that
\be
\sum_{N=1}^{\infty} \PP (\e^{-1} e^{\e^{-1} C_2} W_1(\mu^N_\e(0), f_{0,\e}) \geq C N^{- \alpha}) < \infty .
\ee
Note that in these formulas $\e=\e_N$ depends on $N$ and it satisfies \eqref{epsrate:1d}.
To do this we estimate 
\be
P_n : = \PP ( \e^{-1} e^{\e^{-1} C_2} W_1(\mu^N_\e(0), f_{0,\e}) \geq C N^{- \alpha} )
\ee
using Theorem~\ref{thm_FG} combined with the scaling argument from section~\ref{scaling}.

By assumption, all the $f_{0,\e}$ have compact support contained in $[-1,1] \times B_R$ for some fixed $R > 0$ independent of $\e$. Thus $\s_R[f_{0,\e}]$ (abusing notation) gives a density function supported in $[-1,1]^2$. Observe that Lemma~\ref{lem_scale} implies that 
\be
P_N = \PP ( \e^{-1} e^{\e^{-1} C_2} W_1(\mu^N_\e(0), f_{0,\e}) \geq C N^{- \alpha} ) \leq \PP ( \e^{-1} e^{\e^{-1} C_2} R W_1(\s_R[\mu^N_\e(0)], \s_R[f_{0,\e}]) \geq C N^{- \alpha} ) .
\ee
Observe that $\s_R[\mu^N_\e(0)]$ is a random measure with the same law as the random measure constructed by taking the empirical measure of $N$ independent samples from $\s_R[f_{0,\e}]$. Hence we may apply Theorem~\ref{thm_FG} with $m=2$ and $p=1$ to the right hand side above. Choosing $C = \kappa R$, where $\kappa$ is the constant defined in the statement of Theorem~\ref{thm_FG}, we obtain
\be
P_N \leq \exp{\left (-cN \left (\frac{\e e^{- \e^{-1} C_2} N^{- \alpha}}{\log{(2 + \e^{-1} e^{ \e^{-1} C_2} N^{\alpha})}} \right )^2 \right)} .
\ee

For any choice of $\xi > 0$, there exists a constant $C = C_\xi$ such that
\be
\log{(2 + x)} \leq C x^\xi .
\ee
Hence
\be
P_N \leq \exp{ \left (- C N \left ( \e e^{- \e^{-1} C_2} N^{- \alpha} \right )^{2(1 + \xi)} \right)} .
\ee

Now observe that \eqref{epsrate:1d} implies that
\be
e^{- \e^{-1} C_2} \geq N^{- C_2/A} .
\ee
Moreover for any $\eta > 0$ there exists $C = C_{\eta,C_2}$ such that
\be
\e \geq C e^{- C_2 \e^{-1} \eta} .
\ee
Hence
\be
\left ( \e e^{- \e^{-1} C_2} N^{- \alpha} \right )^{2(1 + \xi)} \geq C N^{- 2(1 + \xi)\left(\frac{C_2}{A} (1 + \eta) + \alpha\right)}
\ee
and thus
\be
P_N \leq \exp{ \left (- C N^{1 - 2(1 + \xi)(\frac{C_2}{A} (1 + \eta) + \alpha)} \right )} .
\ee

The proof is complete if we can choose $\alpha, \eta, \xi$ such that
\be \label{rel_param}
2(1 + \xi) \left (\frac{C_2}{A} (1 + \eta) + \alpha \right ) < 1 .
\ee
This can be done by choosing 
\be
0 < \eta = \xi < \sqrt{\frac{A}{2 C_2}} - 1,
\ee
which is possible since $A > 2 C_2$. Then by construction
\be
2(1 + \xi) \frac{C_2}{A} (1 + \eta) < 1
\ee
so it is possible to choose an $\alpha > 0$ satisfying the given relation \eqref{rel_param}. Thus $\sum_N P_N$ is finite for this choice of $\alpha$, which completes the proof.

\end{proof}

\subsection{Typicality in higher dimensions}

\begin{proof}[Proof of Proposition~\ref{Prop_typ_high}]

We consider
\be
P_N : = \PP \left(W^2_2(\mu^N_\e (0), f_{0,\e}) \geq C \e^{- 2 \gamma} r^{d + 2 + \eta}\right) .
\ee
Our goal is to show that $\sum_N P_N$ is finite for some choice of $C$ and $\eta$; this will imply the desired result by the Borel-Cantelli lemma. We wish to estimate $P_N$ using the concentration inequality and scaling argument described above.

By Lemma~\ref{lem_scale},
\be
P_N \leq  \PP \left(\e^{- 2 \gamma} W^2_2(\s_{\e^{- \gamma}} [\mu^N_\e (0)], \s_{\e^{- \gamma}}[f_{0,\e}]) \geq C \e^{- 2 \gamma} r^{d + 2 + \eta}\right) .
\ee
Since $\s_{\e^{- \gamma}} [\mu^N_\e (0)]$ has the same law (as a random measure) as the empirical measure of $N$ independent samples from $\s_{\e^{- \gamma}}[f_{0,\e}]$, and the scaled measures have support contained in $[-1, 1]^{2d}$, we may apply Theorem~\ref{thm_FG} to deduce the existence of constants $\kappa, c, C$ depending only on $d$ such that
\be \label{conc_data}
\PP \left(W_2^2 (\s_{\e^{- \gamma}}[\mu^N_{0, \eps}], \s_{\e^{- \gamma}}[f_{0, \eps}]) \geq \kappa r^{d + 2 + \eta} \right) \leq C \left \{ \begin{array}{cc}
\exp{\left(-cN \left(\frac{r^{4 + \eta} }{\log{(2 + r^{-(4 + \eta)} }}\right)^2\right)} & \text{for } d=2 \\
\exp{\left(-cN r^{3(5 + \eta)}\right)} &  \text{for } d=3 .
\end{array} \right.
\ee

Using this we deduce that 
\be
P_N \leq C \left \{ \begin{array}{cc}
\exp{\left(-cN^{8 \alpha + \eta\left(2 \alpha - \frac{1}{4}\right)} \left({\log{\left(2 + N^{- \frac{1}{2} + 4 \alpha + \eta\left(\alpha - \frac{1}{8}\right)}\right) }}\right)^{-2}\right)} & \text{for } d=2 \\
\exp{\left(-cN^{15 \alpha + \eta(3 \alpha - \frac{1}{5}) }\right)} &  \text{for } d=3 .
\end{array} \right.
\ee

Thus $\sum_N P_N$ is finite for 
\be
\eta < \alpha \frac{d(d+2)^2}{1 - \alpha d(d+2)},
\ee
which completes the proof.

\end{proof}

 {\it Acknowledgments:}   The authors would like to thank Cl\'ement Mouhot for interesting discussions during the course of this project and the second author is also grateful to Fran\c cois Golse for helpful discussions on the mean field limit for the Vlasov-Poisson equation. The first author was supported by the UK Engineering and Physical Sciences Research Council (EPSRC) grant EP/L016516/1 for the University of Cambridge Centre for Doctoral Training, the Cambridge Centre for Analysis. The second author would also like to acknowledge the L'Or\'eal Foundation for partially supporting this project by awarding the L'Or\'eal-UNESCO \emph{For Women in Science France fellowship}.

\bibliographystyle{abbrv}
\bibliography{M-M-3}

\end{document}